\documentclass[11pt,reqno]{amsart}

\usepackage[utf8]{inputenc}
\usepackage{amsfonts,latexsym,amssymb,amsmath,amsthm,amsrefs,etex,slashed,cancel}
\usepackage{todonotes}
\usepackage{hyperref}
\usepackage{a4wide}
\usepackage{enumerate}
\usepackage{MnSymbol}
\usepackage{nicefrac}

\renewcommand{\Re}{\operatorname{Re}}
\renewcommand{\Im}{\operatorname{Im}}
\newcommand{\dom}{\mathrm{odom}}
\DeclareMathOperator{\supp}{supp}
\newcommand{\der}{\mathrm{d}}
\newcommand{\rmi}{\mathrm{i}}
\newcommand{\ee}{\mathrm{e} }
\newcommand{\sphere}{{\mathbb{S}^{d-1}}}

\newcommand{\tr}{\mathrm{tr}}
\newcommand{\Tr}{\mathrm{Tr}}
\newcommand{\curl}{\mathrm{curl}}
\renewcommand{\div}{\mathrm{div}}
\renewcommand{\dom}{\mathrm{dom}}
\newcommand{\e}{\mathsf{e}}

\newcommand{\Reell}{\mathbb{R}}
\newcommand{\R}{\mathbb{R}}
\newcommand{\C}{\mathbb{C}}

\newcommand{\comp}{0}
\newcommand{\compp}{\mathrm{comp}}
\newcommand{\CnI}{C^\infty_{0}}
\newcommand{\loc}{\mathrm{loc}}
\newcommand{\rel}{\mathrm{rel}}

\newtheorem{theorem}{Theorem}[section]
\newtheorem{definition}[theorem]{Definition}

\newtheorem{lemma}[theorem]{Lemma}
\newtheorem{assumption}[theorem]{Assumption}

\newtheorem{proposition}[theorem]{Proposition}

\newtheorem{rem}[theorem]{Remark}

\title[Birman-Krein formula for Maxwell's equations]{The Birman-Krein formula for differential forms and electromagnetic scattering}

\author[A. Strohmaier]{Alexander Strohmaier}
\address{School of Mathematics,  University of Leeds,  Leeds , Yorkshire, LS2 9JT,
UK} \email{a.strohmaier@leeds.ac.uk}
\thanks{Supported by Leverhulme grant RPG-2017-329}

\author[A. Waters]{Alden Waters}
\address{ University of Groningen, Bernoulli Institute,
Nijenborgh 9,
9747 AG Groningen,
The Netherlands} \email{a.m.s.waters@rug.nl}

\begin{document}

\begin{abstract}
 We consider scattering theory of the Laplace Beltrami operator on differential forms on a Riemannian manifold that is Euclidean near infinity. Allowing for compact boundaries of low regularity we prove a Birman-Krein formula on the space of co-closed differential forms. In the case of dimension three this reduces to a Birman-Krein formula in Maxwell scattering.
\end{abstract}

\maketitle


\section{Introduction and main theorems}

Let $(X,g)$ be an oriented complete connected Riemannian manifold of dimension $d \geq 2$ which is Euclidean near infinity.
This means that there exists a compact subset $K \subset X$ and $R>0$  such that $X \setminus K$ is isometric to $\Reell^d \setminus \overline{B_{R(0)}}$.
Let $\Omega$ be an open subset in $K$ with compact closure  and define $M = X \setminus \Omega$. We will assume throughout that $M$ is connected. The interior of $M$ is then  $X \setminus \overline{\Omega}$ and will be denoted by $M^\circ$.
The subset $\Omega$ will be thought of as an (or many) obstacle(s) in $X$. 

In this paper we will be discussing the scattering of differential forms in $X$ relative to Euclidean space. Scattering takes place because of the possibly non-trivial geometry or topology in $K$ and the possible presence of the obstacles.
One of the important theorems in scattering theory is the Birman-Krein formula relating the spectral shift to the scattering matrix. We refer to the standard textbook \cite{MR1180965} and also \cite{DZ} for background on scattering theory.
Scattering theory of differential forms has been discussed in detail in our paper \cite{OS}, and a Birman-Krein formula has been proved for the Laplace-Beltrami operator on differential forms in this setting with the additional assumption that all the obstacles $\Omega$ have smooth boundary. We extend the class of $\Omega$ for which these results are valid to a much more general class. This includes examples of Lipschitz domains in $\R^d$. The main purpose of this paper however is to give a different trace formula that formally corresponds to taking the trace on the sub-space of co-closed forms. One can view this as related to the Birman-Krein formula, but it does not directly reduce to the formula in \cite{OS}. This is an important general setting for differential forms that will include Maxwell's equations as a special case. Maxwell's equations are vector valued and developing a trace formula for the corresponding time-harmonic evolution operator in the presence of obstacles is more difficult than for the standard Helmholtz equation. 

Let as usual $\der: \CnI(X;\Lambda^\bullet T^{*} X) \to \CnI(X;\Lambda^\bullet T^{*} X)$ be the differential on smooth forms. To cover very general situations will also choose a Hermitian bundle metric on the vector-bundle of differential forms $\Lambda^\bullet T^*X$, but require that
\begin{enumerate}
 \item on $X \setminus K$ the Hermitian metric coincides with the usual Euclidean-induced bundle,
  \item $\Lambda^p T^*_x X$ is orthogonal to $\Lambda^q T^*_x X$ if $p \not= q$,
 \item if $\e(\xi) : \Lambda^\bullet T^*_x X \to \Lambda^\bullet T^*_x X$ is the operator of exterior multiplication by $\xi \in T^*_x X$ then $\ker{\e(\xi)} \cap \ker{\e(\xi)^*} = \{0\}$ whenever $\xi \not =0$. Here $\e(\xi)^*$ is the adjoint of $\e(\xi)$ with respect to the bundle metric.
\end{enumerate}

These conditions are obviously satisfied for the bundle metric induced by the metric $g$. The space $L^2(X,\Lambda^\bullet T^{*} X)$ will be equipped with the inner product constructed from the bundle metric and the metric volume form. Formal adjoints will be taken with respect to this inner product throughout.
Let $\delta: \CnI(X;\Lambda^\bullet T^* X) \to \CnI(X; \Lambda^\bullet T^{*} X)$ be
the formal adjoint of $\der$ with respect to this bundle metric.
The generalised Laplace-Beltrami operator $\Delta$ on differential forms is defined as $\Delta= \der \delta + \delta \der$. 
Note that outside $K$ this operator agrees with the usual Laplace operator. Since we allow general bundle metrics to form $\delta$ the operator $\Delta$ does in general not have scalar principal symbol. 
The conditions guarantee that the operator $\Delta$ has principal symbol that is symmetric and positive definite. In particular, condition (3) implies that $\Delta$ is elliptic and the weak unique continuation property holds, as will be explained below. 

The case of forms of degree one is of particular interest in scattering theory of the electromagnetic field in dimension $d=3$.
For an electric field given by $D, E \in C^\infty(\R_t \times M^\circ,T^*M)$ and a magnetic field given by $\underline{H}, \underline{B} \in C^\infty(\R_t \times M^\circ,T^*M)$. Maxwell's equations in linear matter are
\begin{gather*}
 \curl E = - \dot B, \quad \div D =0,\\
 \div \underline{B} = 0, \quad   \curl\, \underline{H} = \dot D,\\
 D = \mathbf{\epsilon}  E, \quad \underline{B} = \mathbf{\mu} \underline{H},
\end{gather*}
where we use the dot for the $t$-derivative, i.e. $\dot E = \partial_t E$. Here $\epsilon$ and $\mu$ incorporate the effect of matter and are positive-definite matrix-valued functions which we assume to be smooth. Metallic boundary conditions in the case of smooth obstacles correspond to the tangential component of $E$ and the normal component of $\underline{B}$ vanishing at the boundary. To connect this to the language of differential forms one considers instead of the co-vector field $\underline{B}$ the two-form $B$ defined by $\underline{B}= * B$, where $*$ is the Hodge-star operator on $M$.
We then think of $\epsilon$ as a smooth $\mathrm{End}(\Lambda^1 T^*M)$-valued function on $M$ which is pointwise positive definite. Similarly, $\mu$ will be thought of as a smooth $\mathrm{End}(\Lambda^2 T^*M)$-valued function on $M$. One can now define a metric on $\Lambda^p T_x^*M$ by
$$
 (v, w) = \begin{cases} v w, \; & p=0,\\ (v, \epsilon \,w)_g, & p=1,\\ (v, \mu^{-1} w)_g, & p=2,\\ (v, w)_g,  & p=3. \end{cases}
$$
where $v,w \in \Lambda^p T_x^*M$ and $(\cdot,\cdot)_g$ is the metric induced inner product on $\Lambda^p T_x^*M$. This defines a bundle metric on the direct sum $\Lambda^\bullet T_x^*M$ and one computes
\begin{gather*}
 \div (\epsilon E) = -\delta E,\quad * \der E = \mathrm{curl} E,\quad \delta (\mu H) = \epsilon^{-1}\mathrm{curl}\underline{H}.
\end{gather*}
Here the operator $\mathrm{curl}$ acts on one forms and is defined as $* \der$. This agrees with the usual $\mathrm{curl}$ operator on vector fields if they are identified using the musical isomorphism induced by the metric $g$.
We therefore obtain for Maxwell's equations the differential form version
\begin{gather*}
 \der E = - \dot B, \quad \der B =0,\\
 \delta E = 0, \quad  \delta  H = \dot E.
\end{gather*}
This completely absorbs the effect of matter into the bundle metric.
As before $E$ and $B$ are time-dependent differential forms.
Metallic boundary conditions now mean that $E$ and $B$ have vanishing tangential components, so that $\underline{B}$ has vanishing normal component.
In particular for $E$ one obtains the equation
$$
 \epsilon \ddot E + \curl (\mu^{-1} \curl E) = 0, \quad \div (\epsilon E) = 0,
$$
or equivalently
$$
   \ddot E + \Delta E =0, \quad \delta E = 0.
$$
One is therefore lead to the spectral theory of the operator $\Delta$ on divergence-free covector fields satisfying appropriate boundary conditions, where the effect of matter is hidden in the operator $\delta$. In case $\epsilon=\mu=1$ one has as usual $\Delta = \curl \,\curl  -  \mathrm{grad} \;\div$.

For general dimension and arbitrary form-degree the generalisation of this is the spectral theory of the generalised Hodge Laplacian $\Delta$ on co-closed $p$-forms with relative boundary conditions. The construction of a self-adjoint extension of this generalised Hodge-Laplacian is easier to state if one considers the operator as acting on the bundle of all forms, keeping in mind that $\Delta$ preserves the form degree.
First define the (unbounded) operator $\der_c: C^\infty_0(M^\circ, \Lambda^\bullet T^{*}M) \to L^2(M^\circ, \Lambda^\bullet T^{*}M)$.
Its adjoint $\der^*_c$ has domain 
$$
 \mathrm{dom}(\der^*_c) = \{ f \in L^2(M^\circ, \Lambda^\bullet T^{*}M) \, \mid \, \delta f \in L^2(M^\circ, \Lambda^\bullet T^{*}M) \},
 $$ where the derivatives and co-derivatives are in the sense of distributions on $M^\circ$.
In particular, $\der^*_c$ is densely defined and therefore the adjoint $(\der^*_c)^*$ coincides with the closure of $\der_c$.

We will now denote by $\overline{\der}$ the closure of  $\der_c$. 
As a consequence of $\der^2=0$ it follows from abstract theory that $D = \overline{\der} + \overline{\der}^*$, with implied domain $\mathrm{dom}(D)=\mathrm{dom}(\overline{\der}) \cap \mathrm{dom}(\overline{\der}^*)$, is automatically a self-adjoint operator, and so is its square
$$
 \Delta_{\rel} = D^2 = \overline{\der}\, \overline{\der}^* + \overline{\der}^*\, \overline{\der}.
$$
This operator is called the Laplace operator with {\sl relative boundary conditions}. We refer to Appendix \ref{relap} for details. 

Another relevant operator is the absolute Laplacian $\Delta_{\mathrm{abs}}$ which can be defined by
$\Delta_{\mathrm{abs}} = \tilde{*}^{-1} \Delta_{\mathrm{rel}} \tilde{*}$, where $\tilde{*}$ is the generalised Hodge star satisfying
$\langle v, w \rangle \mathrm{dVol}_g = v \wedge \tilde{*} w.$
The operator $\Delta_{\mathrm{abs}}$  can also be constructed in the above manner by interchanging the roles of $\der$ and $\delta$. Since these operators are related immediately by the generalised Hodge star operator we will discuss only relative boundary conditions in this paper and remark that results for the absolute Laplacian are obtained easily by conjugating with the generalised Hodge star operator. \\
Note that the above definition of the Hodge Laplace operator makes sense for any smooth manifold. In particular it also makes sense for any open subset $\mathcal{O} \subset X$. We denote the corresponding operator by $\Delta_{\mathcal{O},\mathrm{rel}}$. Define $\Omega_R$ as $K^\circ \setminus \overline \Omega$. Throughout we will make the following assumption.
\begin{assumption} \label{assumeit}
 There exists a $k>0$ such that $(\Delta_{\Omega_R,\rel}+1)^{-1}$ and $(\Delta_{\Omega,\rel}+1)^{-1}$ are in the $k$-th Schatten ideal in $L^2(\Omega_R,\Lambda^\bullet T^{*}X)$ and $L^2(\Omega,\Lambda^\bullet T^{*}X)$, respectively.
\end{assumption}

This is not a severe restriction. It is implied, for example, by any Weyl-type asymptotic for compact domains with the assumed regularity class of $\Omega$. Weyl asymptotics have been established for Lipschitz domains in $\R^d$ \cite{MR3113431}. The condition above is however not empty. Indeed, in form degree $d$ the operator $\Delta_\rel$ is equivalent to the Neumann Laplacian. Examples in dimension two show that this operator can have essential spectrum \cite{simon}.  

We describe now the spectral theory of $\Delta_\rel$ under the above assumption. The spectrum of $\Delta_\rel$ consists of the point $0$, which is an eigenvalue of finite multiplicity, and the absolutely continuous part $[0,\infty)$. This is a consequence of the meromorphic continuation of the resolvent, general stationary scattering theory, and unique continuation. We summarise the relevant construction in Section \ref{scatt} and now describe the spectral resolution.

Choose an orthonormal basis $(u_j)_{j=1,\ldots,N}$ in $\ker(\Delta_{\mathrm{rel}})$ consisting of eigensections with eigenvalue zero.
Then $(u_j)$ gives the discrete part of the spectrum.  The continuous part of the spectrum is described by the generalised eigenfunction
$E_\lambda(\Phi) \in C^\infty(M^\circ; \Lambda^\bullet T^{*}M)$ that are indexed by $\Phi \in C^\infty(\sphere;\Lambda^\bullet \C^d)$ and $\lambda>0$. We refer to Section 4 for the construction of $E_\lambda$. However, in order to define the main notions we record here its defining properties.

\begin{proposition}
For fixed $\lambda>0$ the generalised eigenfunctions $E_{\lambda}(\Phi)$ are completely determined by the following
\begin{enumerate}
\item $(\Delta- \lambda^2 ) E_\lambda(\Phi) =0$,
\item $\chi E_\lambda \in \mathrm{dom}(\Delta_\rel)$ for any $\chi \in C^\infty_0(M)$ with $\der \chi=0$ near $\partial \Omega$,
\item The asymptotic expansion
$$
 E_\lambda(\Phi) = \frac{\ee^{-\rmi \lambda r} \ee^{\frac{i\pi(d-1)}{4}} }{r^{\frac{d-1}{2}}} \Phi + \frac{\ee^{\rmi \lambda r} \ee^{-\frac{i\pi(d-1)}{4}} }{r^{\frac{d-1}{2}}} \Psi_\lambda + O\left(\frac{1}{r^{\frac{d+1}{2}}}\right),  \quad \textrm{for} \,\,\,r \to \infty.
$$
for some $\Psi_\lambda \in C^\infty(\sphere;\Lambda^\bullet \C^d)$.
\end{enumerate}
\end{proposition}

As a result $\Psi_\lambda$ is uniquely determined and implicitly defines a linear mapping 
\begin{gather*}
 S_\lambda:  C^\infty(\sphere;\Lambda^\bullet \C^d) \to C^\infty(\sphere;\Lambda^\bullet \C^d), \quad \Phi \mapsto \tau \Psi_\lambda,
\end{gather*}
where $\tau: C^\infty(\sphere;\Lambda^\bullet \C^d) \to C^\infty(\sphere;\Lambda^\bullet \C^d) $ is the pull-back of the antipodal map. The map $S_\lambda : C^\infty(\sphere,\Lambda^\bullet \C^d) \to C^\infty(\sphere,\Lambda^\bullet \C^d)$ is called the scattering matrix,
and $ A_\lambda= S_\lambda - \mathrm{id}$ is called the scattering amplitude. Reminiscent of the Hodge-Helmholtz decomposition the scattering matrix admits a decomposition
$$
 S_\lambda = \left( \begin{matrix} S_{n,\lambda} & 0 \\ 0 & S_{t,\lambda} \end{matrix} \right),
$$
if $C^\infty(\sphere,\Lambda^\bullet \C^d)$ is decomposed as $C^\infty_n(\sphere,\Lambda^\bullet \C^d) \oplus C^\infty_t(\sphere,\Lambda^\bullet \C^d)$ into normal and tangential parts. Here $C^\infty_n(\sphere,\Lambda^\bullet \C^d)$ is the kernel of the map
$\der r \wedge $, where $dr$ is the unit conormal on the sphere, and $C^\infty_t(\sphere,\Lambda^\bullet \C^d)$ is the image of $\iota_{\der r}$, inner multiplication by $\der r$.

The spectral shift function usually describes the trace of the difference of functions of perturbed and unperturbed operators in scattering theory.
In our setting these operators act on different Hilbert spaces so a suitable domain decomposition is needed.
Let $P$ be the orthogonal projection $L^2(M^\circ) \to L^2(M \setminus K)$, and let  $P_0$ be the orthogonal projection
$L^2(\R^d) \to L^2(\R^d \setminus B_R(0))$.
Note that on $L^2(\R^d, \Lambda^\bullet \C^d)$ we have the {\sl free Laplacian} $\Delta_0$ with domain $H^2(\R^d, \Lambda^\bullet \C^d)$. Our main result is the following.

\begin{theorem} \label{main1}
 Let $f \in C^\infty_0(\R)$ be an even compactly supported smooth function and let $0 \leq p \leq d$. Let $Q$ be either $\delta \der$ or $\der \delta$ regarded as a differential operator. Then the operators
 $$
  (1-P) Q f(\Delta_\rel^{1/2}) (1-P), \;(1-P_0) Q f(\Delta_0^{1/2}) (1-P_0), P  Q f(\Delta_\rel^{1/2})P - P_0 Q f(\Delta_0^{1/2}) P_0
 $$
 are trace-class and
 \begin{gather*}
  \Tr_p \left( (1-P) Q f(\Delta_\rel^{1/2}) (1-P) \right) - \Tr_p \left( (1-P_0) Q f(\Delta_0^{1/2}) (1-P_0) \right) \\+ \Tr_p \left( P Qf(\Delta^{1/2}_{\rel})P - P_0 Q f(\Delta_0^{1/2}) P_0 \right) = \frac{1}{2 \pi \rmi }\int_{0}^\infty\lambda^2  f(\lambda) \Tr_{L^2(\sphere,\Lambda^p \C^d)} \left ( S_Q^*(\lambda) S_Q^{\;\prime}(\lambda) \right) \der \lambda,
 \end{gather*}
 where we set $S_Q(\lambda):= S_{t,\lambda}$ if $Q=\delta \der$, and $S_Q(\lambda):= S_{n,\lambda}$ in case $Q=\der \delta$.
 Here $\Tr_p$ denotes the trace of the operators on the subspace of $p$-forms.
\end{theorem}

\begin{rem}
It is tempting to try to reduce this statement to the Birman-Krein formula for forms, we do not however see an easy way to achieve that directly.
The Helmholtz-Hodge-Kodaira decomposition depends heavily on the operator and the boundary conditions. Indeed,
if $\Pi$ denotes the orthogonal projection onto $\overline{\mathrm{rg}(\delta)}$ and $\Pi_0$ the orthogonal projection onto 
$\overline{\mathrm{rg}(\delta_0)}$ then the operator  $\Pi - \Pi_0$ is in general not  trace-class if the boundary is smooth and non-empty. This can be seen by analysing the singular behaviour of the diagonal of its integral kernel near the boundary.
It also follows that $(\delta \overline d+1)^{-N} - (\delta_0 \overline d_0+1)^{-N}$ is not in general a trace-class operator for $N> \frac{d}{2}$.
\end{rem}

In case $X = \R^3$ this becomes a statement about Maxwell  obstacle scattering. We were not able to locate such a statement in the literature even in that special case.
To formulate this it is convenient to define the operator $\Delta_{\rel,\R^3}$ on one forms as an operator on $L^2(\R^3,\C^3) = L^2(M^\circ,\C^3) \oplus L^2(\Omega,\C^3) \oplus L^2(\partial \Omega,\C^3)$
as the direct sum $\Delta_{\rel,M} \oplus \Delta_{\rel,\Omega} \oplus 0$, where $\Delta_{\rel,\Omega}$ is the relative Laplacian on the interior domain $\Omega$. The spectrum of $\Delta_{\rel,\Omega}$ then consists of the Dirichlet eigenvalues $(\lambda_j^2), \lambda_j \geq 0$ of the scalar Laplacian on $\Omega$ and the Maxwell eigenvalues $(\mu_j^2), \mu_j \geq 0$.
The Maxwell eigenvalues are the eigenvalues of $\curl\, \curl$ on the space of divergence free one-forms. We did not assume that the boundary $\partial \Omega$ of $\Omega$ has vanishing Lebesgue measure. We therefore
define $\curl\,\curl$ to be zero on $L^2(\partial \Omega,\C^3)$ so that the operator
  $\curl\,\curl f(\Delta_{\rel,\R^3}^{1/2})$ on $L^2(\R^3,\C^3) = L^2(M^\circ,\C^3) \oplus L^2(\Omega,\C^3) \oplus L^2(\partial \Omega,\C^3)$
  is a direct sum $\curl\,\curl f(\Delta_{\rel,M}^\frac{1}{2}) \oplus \curl\,\curl f(\Delta_{\rel,\Omega}^\frac{1}{2}) \oplus 0$.
 Here, potentially confusingly, $L^2(\partial \Omega,\C^3)$ is defined with respect to the Lebesgue measure on $\R^3$. In particular $L^2(\partial \Omega,\C^3)=0$ if $\Omega$ has Lipschitz boundary.

\begin{theorem} \label{B2}
 Let $f \in C^\infty_0(\R)$ be an even compactly supported smooth function. 
 Then,
 $$
   \curl\,\curl f(\Delta_{\rel,\R^3}^{1/2}) -  \curl\,\curl f(\Delta_0^{1/2})
 $$
 is trace-class as on operator on $L^2(\R^3,\C^3)$ and its trace equals
 \begin{gather*}
 \Tr(\curl\,\curl f(\Delta_{\rel,\R^3}^{1/2}) -  \curl\,\curl f(\Delta_0^{1/2}))\\=
 \frac{1}{2 \pi \rmi }\int_{0}^\infty\lambda^2  f(\lambda) \Tr_{L^2(\mathbb{S}^2,\C^3)} \left ( S_t^*(\lambda) S_t^{\;\prime}(\lambda) \right) \der \lambda + \sum_{j=0}^\infty \mu_j^2 f(\mu_j).
 \end{gather*}
 \end{theorem}
 
\begin{rem}
 The function class of compactly supported smooth even functions in Theorems \ref{main1} and \ref{B2} can be extended by continuity if both sides of the equality are continuous on a larger space of functions. For example, the known Weyl laws for the scattering phase in case of smooth boundary imply immediately that the formulae also hold for even Schwartz functions (see Remark 6.3 in \cite{OS}) or symbols $f \in S^{-k}(\R)$ of sufficiently negative order $-k$ that are even, as an immediate consequence.
 \end{rem}

The paper is organised as follows. In Section \ref{scatt} we collect the elements of stationary scattering theory needed that are needed for the proof of the main result. The main theorem is proved in Section \ref{secproof}.  In Appendix \ref{relap} we recall the definition and very general basic properties of the relative Laplace operator. Appendix \ref{vecspher} summarises some simple but important properties of vector-spherical harmonics in arbitrary dimension.
   
\subsection{Notational convention} 

All functions in this paper are complex-valued, unless stated otherwise. This means $L^p(X)$ means $L^2(X,\C)$
and $C^\infty(X)$ means $C^\infty(X,\C)$. We will also assume that $p$-forms are complex-valued, but we will be writing $C^\infty(X; \Lambda^\bullet T^*X)$ for  $\Lambda^\bullet T^*X = \Lambda_\C^\bullet T^*X = \Lambda^\bullet T_\C^*X$, mildly abusing notations. Similar conventions will be used for $L^2$ and Sobolev spaces. Inner products are assumed to be conjugate linear in the second argument.  The open unit ball in $\R^d$ centered at zero of radius $\rho>0$ will be denoted by $B_{\rho}$.


\section{Stationary scattering theory and the spectral resolution} \label{scatt}

In this section we collect the basic facts about the spectral theory of $\Delta_{\rel}$ as it follows from stationary scattering theory. For general background on the theory of black-box scattering for functions and current developments we refer to \cite{MR1451399} and the recent monograph \cite{DZ}.

Since $M \setminus K$ is isometric to $\R^d \setminus B_R(0)$ we have a natural coordinate system 
on $M \setminus K$. We will use both Cartesian coordinates $x \in \R^d$ and spherical coordinates
$(r,\theta) \in (R,\infty) \times \mathbb{S}^{d-1}$, where $r=|x|$ and $\theta= \frac{x}{|x|}$, where it is understood. We choose a smooth function $\chi \in C^\infty(M)$
supported in $M \setminus K$ such that $1-\chi$ is compactly supported. Using the Cartesian coordinates and the orthonormal frame $(\der x^1,\ldots,\der x^d)$
we trivialise the bundle $T^*(M \setminus K)$ and thereby identify forms in $C^\infty(M \setminus K; \Lambda^p T^*M)$ with vector-valued functions
in $C^\infty(M \setminus K; \Lambda^p \C^d)$. We will now assume for notational convenience that $\tilde K$ is a large ball of radius $R>0$ so that $M \setminus K$ is identified with $\R^d \setminus B_R(0)$. This way $M$ is decomposed into 
$\R^d \setminus B_R(0)$ and $K^\circ \setminus \Omega$.

Let us make same remarks about the domain of $\Delta_\rel$. First the definition of a Sobolev space makes sense on $X$. Namely, $f \in H^s(X, \Lambda^\bullet T^{*} X)$ if and only if $\chi f \in H^s(\R^d,\Lambda^\bullet \C^d)$ and $(1-\chi) f \in H^s_{\mathrm{loc}}(X)$ if $\chi$
is a smooth cut-off function supported in $X \setminus K$ that equals one outside a compact set. Then the space $C^\infty_0(X,\Lambda^\bullet T^{*} X)$ is dense in $H^s(X,\Lambda^\bullet T^{*} X)$. Using the explicit description of the domain of $\overline d^*$ and the domain of $\overline d$ as the closure of $C^\infty_0(M^\circ,\Lambda^\bullet T^{*} M)$ in the graph norm of $d_c$ it is not hard to show the following proposition.

\begin{proposition} \label{sobain}
 Suppose that $\eta \in C^\infty(M)$ vanishes near $\partial \Omega$ and $1 - \eta \in C^\infty_0(M)$. Then the following statements hold.
 \begin{itemize}
  \item If $\phi \in L^2(M,\Lambda^\bullet T^{*} M)$ and $\Delta\phi \in L^2(M, \Lambda^\bullet T^{*} X)$ imply that $\eta \phi \in \mathrm{dom}(\Delta_\rel)$.
  \item If $\phi \in  \mathrm{dom}(\Delta_\rel)$ then $\eta \phi \in H^2(X,\Lambda^\bullet T^{*} X)$.
 \end{itemize}
\end{proposition}
\begin{proof}
 For the first statement note that $\phi \in L^2(M, \Lambda^\bullet T^{*} X), \Delta\phi \in L^2(M, \Lambda^\bullet T^{*} X)$ implies in particular that $\eta \phi \in H^2(X, \Lambda^\bullet T^{*} X)$. Now $C^\infty_0(X, \Lambda^\bullet T^{*} X)$ is dense in $H^2(X, \Lambda^\bullet T^{*} X)$, but since the support of $\eta \phi$ has positive distance from $\Omega$
 an approximating sequence can be chosen to be in $C^\infty_0(M^\circ, \Lambda^\bullet T^{*} X)$. Thus
 the approximating sequence is in the domain of the operator $\Delta_\rel$, and convergence for this sequence in $H^2(X, \Lambda^\bullet T^{*} X)$ implies convergence in the graph norm for $\Delta_\rel$. Since $\Delta_\rel$ is closed we have that $\eta \phi$ must be in its domain.\\
 For the second statement we use elliptic regularity. Since $\phi$ in $\mathrm{dom}(\Delta_\rel)$ implies in particular that
 $\Delta\phi \in L^2$ in the sense of distributions we obtain $\phi \in H^2_\loc(M^\circ, \Lambda^\bullet T^{*} X)$
 and then, using the product rule, $\Delta(\eta \phi ) \in L^2(M^\circ, \Lambda^\bullet T^{*} X)$. Again, by elliptic regularity, we obtain
 $\eta \phi \in H^2(X,\Lambda^\bullet T^{*} X)$.
\end{proof}

This statement encodes that the boundary conditions implicitly defined by the domain of the operator $\Delta_\rel$ are local near the boundary $\Omega$ and implies that integration by parts is possible and results in no boundary terms from $\partial \Omega$ in case the functions are locally in the domain of $\Delta_\rel$. 

\begin{proposition} \label{ibpprop}
 Let $U$ be a compact submanifold of $X$ with smooth boundary $\partial U$ such that $\Omega \subset U$ and $\partial U \subset X \setminus K$. Let $V = U \setminus \Omega$.
 Suppose that $u,v \in C^\infty(M^\circ, \Lambda^\bullet T^{*} M)$ such that 
 for any cut-off function $\chi \in C^\infty_0(M)$ that equals one near $\partial \Omega$ the form $\chi u$ and 
 $\chi v$ are in $\mathrm{dom}(\Delta_\rel)$. Then
 $$
  \int_{V} \langle \Delta u , v \rangle_x \der \mathrm{Vol}_g(x)   - \int_V \langle  u, \Delta v \rangle_x \der \mathrm{Vol}_g(x) = \int_{\partial U} \langle  u, \nabla_n v \rangle_x - \langle \nabla_n u, v \rangle_x  \der \sigma_{\partial U}(x),
 $$
 where $\nabla_n$ is the covariant derivative with respect to the outward pointing normal vector field on $\partial U$,
 and $\der \sigma_{\partial U}$ is the surface measure on $\partial U$.
\end{proposition}
\begin{proof}
 This is proved directly by decomposing the functions $u,v$ into functions supported near $\partial U$ and functions in the domain of the operator, using integration by parts and self-adjointness of $\Delta_\rel$. Note that the inner product on $\Lambda^\bullet T^* M$ equals the Euclidean (metric) inner product near $\partial U$.
\end{proof}

\begin{proposition}
 Let $L: C^\infty(M^\circ,\Lambda^\bullet T^*M) \to C^\infty(M^\circ,\Lambda^\bullet T^*M)$ be a first order differential operator and $u \in \mathcal{D}'(M^\circ,\Lambda^\bullet T^*M)$ satisfies $(\Delta + L) u=0$. If $u$ vanishes on a non-empty open subset of $M^\circ$, then $u=0$ on $M^\circ$.
\end{proposition}
\begin{proof}
 By elliptic regularity $u$ is smooth. Let $x$ be in the boundary of the interior of the zero set of $u$. Then $u$ vanishes of infinite order at $x$. Now choose Riemann normal coordinates at $x$ and a bundle-frame that is orthonormal with respect to the bundle metric, so that the principal symbol of $\Delta_{\rel}$ in that frame is diagonal. Then strong unique continuation at that point follows from \cite{MR2361423}. 
Strong unique continuation at every point then implies the weak unique continuation property.
\end{proof}

Since $\Delta_{K,\rel}$ is an elliptic differential operator with compact resolvent the gluing method of black-box scattering (see e.g. \cite{MR1451399}) can be applied. This means that the resolvent $(\Delta_\rel -\lambda^2)^{-1}$
can be written as
$$
 (\Delta_\rel -\lambda^2)^{-1} = \left( \chi_1 (\Delta_0 -\lambda^2)^{-1} \eta_1 + \chi_2 (\Delta_{\Omega_R} -\lambda^2)^{-1} \eta_2 \right) (1 + K_\lambda).
$$
Here $\chi_1,\eta_1,\chi_2,\eta_2$ are suitably defined smooth gluing functions with $\chi_1,\eta_1$ vanishing near $\Omega$
and $1-\chi_1,1-\eta_1,\chi_2,\eta_2$ compactly supported. The family $K_\lambda$ is a meromorphic family of compact operators on a logarithmic cover of the complex plane mapping into functions supported in the gluing region. The negative Laurent coefficients of $K$ are finite rank. The function $\tilde j_\lambda(\Phi)$ is defined in spherical coordinates on $X \setminus K$ by
 \begin{align*} 
\tilde j_\lambda(\Phi)(r \theta) &=2\lambda^{\frac{d-1}{2}} j_{d,\ell}(\lambda r) (-\rmi)^{\ell} \Phi(\theta), \quad
 j_{d,\ell}(x)=\sqrt{\frac{\pi}{2}}x^{\frac{2-d}{2}}J_{\ell + \frac{d-2}{2}}(x),
 \end{align*}
 if $\Phi$ is a spherical harmonic of degree $\ell$ and extends linearly and continuously for general $\Phi \in C^\infty(\sphere,\Lambda^\bullet \C^d)$.
\begin{definition} \label{hankelsums}
 We define $\tilde h^{(1)}_{\lambda}(\Phi)$, and  $\tilde h^{(2)}_{\lambda}(\Phi)$ by 
 \begin{align*}
  \tilde h^{(1)}_\lambda(\Phi)(r \theta) &=\lambda^{\frac{d-1}{2}} \sum\limits_{\nu} \langle \Phi, \Phi_\nu \rangle_{L^2(\sphere)} \Phi_{\nu}(\theta)h^{(1)}_{\ell_{\nu}}(\lambda r) (-\rmi)^{\ell_{\nu}},\\
   \tilde h^{(2)}_\lambda(\Phi)(r \theta) &=\lambda^{\frac{d-1}{2}} \sum\limits_{\nu} \langle \Phi, \Phi_\nu \rangle_{L^2(\sphere)} \Phi_{\nu}(\theta)h^{(2)}_{\ell_{\nu}}(\lambda r) (-\rmi)^{\ell_{\nu}},
 \end{align*}
 whenever the sums converge in $C^{\infty}(\mathbb{R}^{d}\setminus \{0\})$. Here $h^{(1)}_\ell$, and $h^{(2)}_\ell$ are the spherical Hankel functions in dimension $d$ defined as
 \begin{gather*}
  h^{(1)}_{d,\ell}(x)=\sqrt{\frac{\pi}{2}}x^{\frac{2-d}{2}}H^{(1)}_{\ell + \frac{d-2}{2}}(x),\quad
  h^{(2)}_{d,\ell}(x)=\sqrt{\frac{\pi}{2}}x^{\frac{2-d}{2}} H^{(2)}_{\ell + \frac{d-2}{2}}(x).
\end{gather*}
\end{definition}
 Throughout the paper, if $z \not=0$ is an element of the logarithmic cover of the complex plane, we will define $-z = \mathrm{e}^{\rmi \pi} z$ which corresponds to a counterclockwise rotation by $\pi$. Some care is needed with this notation, however, since then $-(-z)$ is on a different sheet than $z$. The complex conjugate of $z= r \ee^{\rmi \phi}$ in the logarithmic cover is defined by $\overline{z} = r \ee^{-\rmi \phi}$. For $z>0$ the complex conjugate
$\overline{-z}$ of $-z$ is then also on another branch than $-z$, namely $\overline{-z} = \ee^{-\rmi \pi} z$.
 The properties of the Hankel functions imply that
 \begin{gather} \label{rotationhankel1}
  h^{(1)}_{d,\ell}(x e^{\rmi \pi}) = - (-1)^{\ell+d} h^{(2)}_{d,\ell}(x),\\   \label{rotationhankel2}
  h^{(2)}_{d,\ell}(x e^{\rmi \pi})=  (-1)^{\ell} \left( h^{(1)}_{d,\ell}(x) + (1+(-1)^d) h^{(2)}_{d,\ell}(x) \right).
\end{gather}
Using a smooth cut-off function $\chi$ that vanishes near $K$ and equals one outside a compact set one constructs the generalised eigenforms as
\begin{align} \label{Econstr}
 E_\lambda(\Phi) = \chi \tilde j_\lambda(\Phi) -  (\Delta_\rel -\lambda^2)^{-1} [\Delta ,\chi] \tilde j_\lambda(\Phi)
\end{align}
This formula, \eqref{Econstr} and the above definition gives the following proposition: 
\begin{proposition} \label{EoutofHankel}
 For every $\lambda \in \R \setminus \{0\}$ and $\Phi \in L^2(\sphere, \Lambda^p\mathbb{R}^d)$ there exists a unique $A_\lambda(\Phi) \in C^\infty(\sphere, \Lambda^p\mathbb{R}^d)$ such that
 \begin{gather*}
  E_{\lambda}(\Phi)|_{M \setminus K} = \tilde{j}_{\lambda}(\Phi) + \tilde h^{(1)}_\lambda(A_\lambda \Phi).
 \end{gather*}
\end{proposition}
c.f. Proposition 2.2 Eq. 4 in \cite{OS} which has now been generalized. Rellich's uniqueness and the unique continuation property of $\Delta$ imply that there are no eigenvalues other than zero for $\Delta_\rel$, and $(\Delta_\rel -\lambda^2)^{-1}$ has no poles in $\R \setminus \{0\}$. The multiplicity of the zero eigenspace is finite. In case of odd dimension this follows from the fact that the resolvent is meromorphic near zero with finite rank negative Laurent coefficients. In case of even dimensions there is a convergent Hahn-expansion, again with finite rank negative expansion coefficients. We refer to \cite{OS} for the proofs carried out for the case when $\Omega$ is smooth. 

As a consequence of the above Propositions the spectral results of \cite{OS} carry over to the setting described in this paper. More precisely, Th. 1.5, 1.6, 1.7, 1.9, 1.10, 1.11 hold in this context. The expansion of Th 1.4, 1.8 also hold with the modification that the $O_{C^\infty(M)}$ terms need to be replaced by $O_{C^\infty(M^\circ)}$. This modification is necessary only because elliptic regularity was used in the proofs, in in our general context this does not hold any more up to the boundary. The Birman-Krein formula Theorem 6.1 of \cite{OS} also holds with the same proof, with the only modification that Theorem 2.13 needs to be replaced by the theorem below.
We only re-state here the results necessary to prove our main result.
Let $\{ u_1,\ldots, u_N \}$ be an orthonormal basis in $\mathrm{ker}(\Delta_\rel)$. We let $(\Phi_\mu)_\mu$ is an orthonormal basis in $L^2(\sphere,\Lambda^p \C^d)$ of spherical harmonics $\Phi_\mu$ of degree $\ell_\mu$.  Then, each eigenfunction $u_j$ admits a multipole expansion 
$$
 u_j = \sum_{\nu} a_{\nu,j} \frac{1}{r^{\ell_{\nu}+d-2}} \Phi_\nu.
$$
For $\Phi \in L^2(\sphere;\Lambda^p \C^d)$ define
$$
 a_j(\Phi) := \sum_{\nu} \overline{a_{\nu,j}} \langle \Phi, \Phi_\nu \rangle,
$$
whenever the sum converges absolutely. In particular the sum is finite when $\Phi$ is a finite linear combination of spherical harmonics. We then have bounds and expansions for the scattering amplitude from \cite{OS} that are summarised in the proposition below.

\begin{proposition} \label{Aone}
There exists an open neighborhood $U_0$ of $\R \setminus \{0\}$ in $\C$ such that the following holds.
Let  $U \subset \C$ be the union of $U_0$ and the  upper half space $\{\lambda \mid \Im(\lambda)>0\}$. Then the scattering matrix $S_\lambda$ extends to a holomorphic function on $U$ with values in the bounded operators $L^2(\sphere, \Lambda^\bullet \C^d) \to L^2(\sphere, \Lambda^\bullet \C^d)$. Recall that $A_\lambda= S_\lambda - \mathrm{id}$ is the scattering amplitude. We then have the following.
\begin{itemize}
 \item[(i)] If $d \geq 3$ is odd, then for any $s \in \R$ the family $A_\lambda$ is a holomorphic family of bounded operators on $U$ with values in $\mathcal{B}(L^2(\sphere),H^s(\sphere))$ and as such extends holomorphically to an open neighborhood of $0$ in $\C$.
 \item[(ii)] If $d \geq 2$ is even the family $A_\lambda$ is a holomorphic family of operators on $U$ 
 with values in $\mathcal{B}(L^2(\sphere),H^s(\sphere))$ for any $s \in \R$. We have for any $s \in \R$ the inequality $\| A_\lambda \|_{L^2\to H^s} = O(|\lambda|^{d-2})$ as $|\lambda| \to 0$ in $U$.
 \item[(iii)] If $d=2$ then $\| A_\lambda \|_{L^2\to H^s} = O(\frac{1}{|\log(\lambda)|})$ as $|\lambda| \to 0$ in $U$.
 \item[(iv)] If $\Phi$ is a spherical harmonic of degree $\ell$, then
 \begin{gather*}
 \langle A_\lambda \Phi, \Phi_\nu \rangle = \left( -\frac{\rmi}{2} (d-2+2\ell) (d-2+2\ell_\nu)C_{d,\ell} \overline{C_{d,\ell_\nu} }\sum_{j=1}^N a_j(\Phi) \overline{a_j(\Phi_\nu)} \right) \lambda^{\ell +\ell_\nu + d-4} + r(\lambda),
  \end{gather*}
  where $r(\lambda) = O(\lambda^{\ell +\ell_\nu + d-4})$ for $|\lambda|<1$ and $C_{d,\ell}$ is defined by
$$
C_{d,\ell} = (-\rmi)^{\ell}\sqrt{2 \pi} \frac{1}{2^{\ell + \frac{d}{2}-1}} \frac{1}{\Gamma(\ell + \frac{d}{2})}.
$$
\item[(v)] There exists a constant $R_1>0$ such that for $\lambda$ in the rectangle $0<\Im(\lambda)<1, -1<\Re(\lambda)<1$ we have the bound
 $$
  | \langle A_\lambda \Phi_\nu, \Phi_\mu \rangle | = O\left(R_1^{\ell_\mu+\ell_\nu} \frac{\lambda^{\ell_\nu + \ell_\mu + d-4}}{\Gamma(\ell_\nu +\frac{d}{2}) \Gamma(\ell_\mu +\frac{d-2}{2})}\right).
 $$
\item[(vi)] We have the following functional relationship for the scattering matrix
$$
 A_{\overline\lambda}^* = (-1)^{d-1} \tau \;A_{-\lambda} \; \tau
$$ 
where $\tau: C^\infty(\sphere, \Lambda^p\mathbb{R}^d) \to C^\infty(\sphere, \Lambda^p\mathbb{R}^d), f(\theta)\mapsto f(-\theta)$ is the pull-back of the antipodal map.
 \end{itemize}
\end{proposition}
We add here some guidance as to where in \cite{OS} these results can be found. In the even dimensional case the neighborhood of $\R \setminus \{0\}$ can normally not be chosen to include zero. In \cite{OS} the corresponding expansions are proved in sectors of the complex plane and the theory of Hahn-holomorphic functions was used to give precise statements about convergence in fixed sectors of the logarithmic cover of the complex plane. We have decided here to not use this language and therefore these neighborhoods will typically exclude the negative imaginary axis. 
Parts (i) and the first part of (ii) are a consequence of Corollary 2.8 in \cite{OS}. The second part of (ii) follows from Theorem 4.2, first case, and by extension (iv) follows from Theorem 1.10 which is proved using Theorem 4.2 and 1.4. The expansions in Theorem 4.2 and Theorem 1.4 in \cite{OS} still hold locally uniformly as a result of Propositions in Section 4. Part (iii) follows from Theorem 4.2 applied with $d=2$. Each of the expansions in the proof is holomorphic in $U$, and can be differentiated. Part (v) follows from Lemma 2.10 in \cite{OS} again since each of the expansions still holds locally uniformly. Part (vi) is a consequence of the definition of $A_{\lambda}$ and Proposition \ref{difftheo} below.

We are going to use only certain components of the spherical harmonics in our relative trace formulae. Therefore we also recall the following result of Proposition 2.6 and equations (7) and (8) from \cite{OS} which still hold in our more general setting. The proof relies on the uniqueness of the generalised eigenforms.
If $\Phi \in C^\infty(\sphere,\Lambda^p \C^d)$ then $\der r \wedge \Phi \in C^\infty(\sphere,\Lambda^{p+1} \C^d)$ and $\iota_{\der r}  \Phi \in C^\infty(\sphere,\Lambda^{p+1} \C^d)$, where $\iota_{\der r}$ is interior multiplication of differential forms by $\der r$. 
\begin{proposition} \label{difftheo}
In case $d$ is odd we have
\begin{gather*}
 E_{-\lambda}(\Phi) = (\rmi)^{d-1 }E_{\lambda}(\tau \; S_{-\lambda} \Phi),\quad
 S_\lambda \; \tau \; S_{-\lambda} =  \tau,
\end{gather*}
and in case $d$ is even we have
\begin{gather*} 
 E_{-\lambda}(\Phi) = (\rmi)^{d-1} E_{\lambda}(\tau ( 2\;\mathrm{id} - S_{-\lambda}) \Phi),\quad
  S_\lambda \; \tau \; (2 \; \mathrm{id} - S_{-\lambda}) = \tau.
\end{gather*}
Moreover, the following equalities hold,
\begin{gather*} 
 \der E_\lambda(\Phi) = -\rmi \lambda E_\lambda(\der r \wedge \Phi),\quad  \delta E_\lambda(\Phi) = \rmi \lambda E_\lambda(\iota_{\der r}\Phi),\\
 \der r \wedge S_\lambda \Phi = S_\lambda \der r \wedge \Phi, \quad  \iota_{\der r} S_\lambda \Phi = S_\lambda \iota_{\der r} \Phi.
\end{gather*} 
\end{proposition}

Since $ \Phi  \mapsto \der r \wedge \iota_{\der r}  \Phi $ is the orthonormal projection onto $L^2_{n}(\sphere,\Lambda^p \C^d)$
and $ \Phi  \mapsto \iota_{\der r} \der r \wedge \Phi$  is the orthonormal projection onto $L^2_{t}(\sphere,\Lambda^p \C^d)$ this proposition implies the claimed splitting of the scattering matrix. We will also need the following bound on the matrix elements of the scattering amplitude, which we recall is a result of Lemma 2.10 in \cite{OS}.

Finally, the basis of eigenfunctions $\{ u_1,\ldots, u_N \}$ in $\mathrm{ker}(\Delta_\rel)$ together with the generalised eigenfunctions $E_\lambda(\Phi)$ provide a complete spectral resolution of $\Delta_\rel$.

The following gives the spectral decomposition in appropriate functions spaces. This was stated in \cite{OS}, (Theorem 2.11) for the case of smooth boundary and the same theorem holds here with the modification that convergence in $C^\infty(M \times M)$ needs to be replaced by convergence in $C^\infty(M^\circ \times M^\circ)$. The reason is that we do not have elliptic boundary regularity in our setting. 

\begin{proposition}\label{kernelconv}
 If $h$ is a Borel function with $h = O((1+\lambda^2)^{-q})$ for all $q \in \mathbb{N}$ we have that $h(\Delta_{p,\mathrm{rel}})$ has smooth integral kernel
 $k_h \in C^\infty(M^\circ \times M^\circ; \Lambda^\bullet T^*M \boxtimes (\Lambda^\bullet T^*M)^*)$ and 
 \begin{gather}
 k_h(x,y) = h(0) \sum_{j=1}^N u_j(x) \otimes (u_j(y))^* \nonumber \\+ \frac{1}{2\pi}  \sum\limits_{\nu}  \int_{0}^\infty h(\lambda^2) E_{\lambda}(\Phi_\nu)(x) \otimes E_{\lambda}(\Phi_\nu)^*(y)\,\der \lambda, \label{eleven}
\end{gather}
where the sum converges in $C^\infty(M^\circ \times M^\circ; \Lambda^\bullet T^*M \boxtimes (\Lambda^\bullet T^*M)^*)$.
\end{proposition}
We briefly outline the adaption of the argument to our setting. If $f \in C_0^\infty(M^\circ,\Lambda^\bullet T^*M)$ and $\lambda>0$ we have
 \begin{align} \label{diffresold}
(R_\lambda-R_{-\lambda})f=\frac{\rmi}{2\lambda}\sum\limits_\nu E_{\lambda}(\Phi_\nu)\langle f, E_{\overline{\lambda}}(\Phi_\nu)\rangle,
\end{align} 
where convergence is in $C^\infty(M^\circ,\Lambda^\bullet T^*M)$. This follows as usual from the characterisation of outgoing and incoming solutions, Rellich's uniqueness theorem, and integration by parts.
Since the resolvent is meromorphic, with a possible pole of order at most two only at zero, one can directly compare with Stone's formula to express the spectral measure $\langle f \der E_\lambda, f \rangle$ in terms of $\langle f, (R_\lambda-R_{-\lambda}) f \rangle$ and $\sum_j |\langle u_j,f \rangle|^2 \delta_0(\lambda)$. Hence, formula \eqref{eleven} holds in the sense of distributions for any bounded Borel function $h$. This can then be improved to convergence in $C^\infty(M^\circ \times M^\circ; \Lambda^\bullet T^*M \boxtimes (\Lambda^\bullet T^*M)^*)$ using the fact that  $ (1 + \Delta_\rel)^s h(\Delta_\rel) (1 + \Delta_\rel)^s$ is $L^2$-bounded for any $s \geq 0$. This uses only the inclusions  $H^s_\comp(M^\circ)  \subset \mathrm{dom}(\Delta_\rel^{s/2}) \subset H^s_\loc(M^\circ)$ for $s \geq 0$, which follow from Proposition \ref{sobain} and elliptic regularity in case $\frac{s}{2} \in \mathbb{N}_0$. 

\section{Proof of the Main Theorems} \label{secproof}

We first need to show that the corresponding operators are trace-class. This is true for 
$(1-P_0)Q f(\Delta_0) (1-P_0)$ since this operator has a smooth integral kernel on the compact manifold $B_R(0)$ with smooth boundary. To show the trace-class property for the two other operators we will use a modification of the gluing method used in \cite{OS} (see \cite{MR1451399} for this method in the context of black-box scattering). We will consider the operator $\Delta_{\Omega_R}$. 
Recall that our assumption implies that $(\Delta_{\Omega_R} + 1)^{-k+1}$ is trace-class for some $k>0$. We assume here w.l.o.g. that $2k > d +1$.
Using
$$
 (\Delta_{\Omega_R} -\lambda^2)^{-1}  = (1 + \lambda^2)(\Delta_{\Omega_R} -\lambda^2)^{-1} (\Delta_{\Omega_R} +1)^{-1} + (\Delta_{\Omega_R}+1)^{-1} 
$$
we conclude that $(\Delta_{\Omega_R} -\lambda^2)^{-1}$ is in the $k$-th Schatten class for $\Im(\lambda)>0$ and the Schatten norm satisfies the uniform bound
$$
  \| (\Delta_{\Omega_R} -\lambda^2)^{-1} \|_k \leq C (1 + | \lambda |^2) \frac{1}{\Im(\lambda^2)}.
$$
Since $(\Delta_{\Omega_R} +1 )^{-1} Q$ is bounded by $1$ in the operator norm as a consequence of Lemma \ref{Tstar} we then see that
$$(\Delta_{\Omega_R} +1)^{-k}  Q (\Delta_{\Omega_R} -\lambda^2)^{-1} $$ is trace-class and its trace norm is bounded by
$$
 \| (\Delta_{\Omega_R} +1)^{-k}  Q (\Delta_{\Omega_R} -\lambda^2)^{-1}   \|_1 \leq C (1 + | \lambda |^2) \frac{1}{\Im(\lambda^2)}.
$$
Next we consider a set of smooth cut-off functions $\chi_1,\eta_1,\chi_2,\eta_2 \in C^\infty(M)$ suitably chosen such that
\begin{gather*}
 \eta_1 \chi_1 = \eta_1, \quad \eta_2 \chi_2 = \eta_2, \quad  \eta_1 + \eta_2 = 1,\\
\supp \chi_1 \subset X\setminus \overline{\Omega}, \quad \supp \chi_2 \subset K,\quad
\mathrm{dist}(\supp \chi_1',\eta_1)>0, \quad \mathrm{dist}(\supp \chi_2',\eta_2)>0.
\end{gather*}
The functions are chosen so that $\eta_1, \chi_1$ vanish near $\overline{\Omega}$ and are equal to one in a neighborhood of $X \setminus K^\circ$. The functions $\eta_2, \chi_2$ are chosen to be one in a neighborhood of 
$\overline{\Omega}$.
Now consider the operator $$T_\lambda = (\Delta_{\rel} +1)^{-k} (\Delta_{\rel} - \lambda^2)^{-1}  Q.$$ Similarly, let
\begin{gather*}
 T_{0,\lambda} = (\Delta_{0} +1)^{-k} Q(\Delta_{0} - \lambda^2)^{-1},\\
 T_{\Omega_R,\lambda} = (\Delta_{\Omega_R} +1)^{-k} Q(\Delta_{\Omega_R} - \lambda^2)^{-1},
\end{gather*}
and define
$$
 \tilde T_\lambda =  \chi_1 T_{0,\lambda} \eta_1 +  \chi_2 T_{\Omega_R,\lambda} \eta_2.
$$
By the support properties of the cut-off functions 
the operator $\tilde T_\lambda$ maps the domain of $Q$ into the domain of $(-\Delta_{\rel} +1)^k (-\Delta_{\rel} - \lambda^2)$.
One then computes $$(\Delta_\rel+1)^k (\Delta_\rel - \lambda^2)  \tilde T_\lambda = Q + R_1 + (\Delta_\rel +1)^k R_2(\lambda),$$ where
\begin{gather*}
 R_1 = [ (\Delta +1)^k, \chi_1 ] (\Delta_{\Omega_R} +1)^{-k}  Q\,\eta_1 +  [ (\Delta+1)^k, \chi_2 ] (\Delta_0 +1)^{-k}  Q \,\eta_2, \\
 R_2(\lambda)=  [\Delta , \chi_1 ] T_{{\Omega_R},\lambda}   \eta_1 +   [\Delta, \chi_2 ] T_{0,\lambda} \eta_2.
\end{gather*}
Therefore, one has
$$
  \tilde T_\lambda = T_\lambda  + (\Delta_\rel +1)^{-k} (\Delta_\rel - \lambda^2)^{-1}  R_1 + (\Delta_\rel - \lambda^2)^{-1}  R_2(\lambda).
$$

By elliptic regularity the resolvent kernels are smooth away from the diagonal. Hence, 
the support properties of the cutoff functions imply that $R_1$ is a smoothing operator mapping to a space of functions with support in a fixed compact set. Hence, $R_1$ is a trace-class operator. For $R_2(\lambda)$ we have
\begin{gather*}
 R_2(\lambda)=  R_3(\lambda) + R_4(\lambda),\\
  R_3(\lambda)=[\Delta , \chi_1 ] (\Delta_{\Omega_R} +1)^{-k}  (\Delta_{\Omega_R} -\lambda^2)^{-1} Q\,\eta_1,\\
  R_4(\lambda) = [\Delta, \chi_2 ] (\Delta_0 +1)^{-k}  (\Delta_0 -\lambda^2)^{-1}  Q \,\eta_2.
\end{gather*}
The operators  $[\Delta, \chi_1 ](\Delta_{\Omega_R} +1)^{-k}$ and $[\Delta, \chi_2 ] (\Delta_0 +1)^{-k}$ continuously map $L^2$ into the space $H^{2k-1}_\compp(U)$ since $2 k>d+1$, where $U$ is a bounded subset of $X$. It follows that these operators are trace-class. The operators $(\Delta_0 -\lambda^2)^{-1} Q$ and $(\Delta_{\Omega_R} -\lambda^2)^{-1} Q$ are both bounded by uniformly by $|\lambda|^2 \Im(\lambda^2)^{-1}$, which is a consequence of spectral calculus and the general inequality $| \frac{x^2}{x^2+z} | \leq \frac{|z|}{\Im(z)}$ for real $x \in \R$. We
conclude that $R_3$ and $R_4$ are trace class for any $\lambda$ in the upper half plane and that
$$\|R_3(\lambda)\|_1 \leq C_3 |\lambda|^2 \Im(\lambda^2)^{-1},\quad  \|R_4(\lambda)\|_1 \leq C_4|\lambda|^2 \Im(\lambda^2)^{-1}.$$ Thus, we obtain
$$
 \| (\Delta_\rel+1)^{-k} (\Delta_\rel - \lambda^2)^{-1}  R_1 + (\Delta_\rel - \lambda^2)^{-1}  R_2(\lambda) \|_1 \leq C  \frac{|\lambda|^2}{|\Im(\lambda^2)|^2}
$$
for some constant $C>0$. We have proved:
\begin{lemma} \label{lemma32}
 The operator $T_\lambda - \tilde T_\lambda$ is trace-class and there exists $C>0$ such that for the trace norm we have
 $$
   \| T_\lambda - \tilde T_\lambda \|_1  \leq C  \frac{|\lambda|^2}{|\mathrm{Im}(\lambda^2)|^2}.
 $$
\end{lemma}

\begin{lemma}\label{tracecornaught}
 For any compactly supported functions $\chi_0 \in L^\infty_{\mathrm{comp}}(\R^d)$ and $\chi \in L^\infty_{\mathrm{comp}}(M)$ the operators $\chi_0  T_{0,\lambda}$ and $\chi  T_{\lambda}$ are trace-class and there exists $C>0$ such that
 $$
  \| \chi_0  T_{0,\lambda} \|_1 + \| \chi  T_{\lambda} \|_1 \leq C\frac{1 + |\lambda|^2}{|\mathrm{Im}(\lambda^2)|^2}.
 $$
\end{lemma}
\begin{proof}
 To show this choose $R>0$ sufficiently large and cut-off functions $\chi_1,\eta_1,\chi_2,\eta_2$ as above such that
 $\chi_2$ and $\eta_2$ are equal one near the support of $\chi$. Then  $ \chi \chi_2 =  \chi$ and $ \chi \chi_1 =0$.
 Thus,
 $$
  \chi\cdot ( T_\lambda ) =   \chi \cdot ( T_\lambda - \tilde T_\lambda) + \chi T_{\Omega_R,\lambda} \eta_2.
 $$ 
 Since $T_{\Omega_R,\lambda}$ is trace-class and $\|T_{\Omega_R,\lambda} \|_1 \leq C_1 \frac{1 + |\lambda|^2}{|\mathrm{Im}(\lambda^2)|^2}$ we obtain 
 $$
  \| \chi  T_\lambda   \|_1 \leq C_2 \frac{1 + |\lambda|^2}{|\mathrm{Im}(\lambda^2)|^2}.
 $$
 The bound on $\| \chi_0 \cdot  T_{0,\lambda}   \|_1$ is immediately implied by this too since this equals 
 $\| \chi\cdot T_\lambda   \|_1$ in the special case $\Omega = \emptyset$.
\end{proof}

\begin{proposition}\label{tracecorone}
The operator $$P \left( (\Delta_{\rel} +1)^{-k} \curl \curl  (\Delta_{\rel} - \lambda^2)^{-1}   -   (\Delta_{0} +1)^{-k} \curl \curl  (\Delta_{0} - \lambda^2)^{-1} \right) P$$ is trace-class and we have the bound
$$
 \| P \left( (\Delta_{\rel} +1)^{-k} \curl \curl  (\Delta_{\rel} - \lambda^2)^{-1}   -   (\Delta_{0} +1)^{-k} \curl \curl  (\Delta_{0} - \lambda^2)^{-1} \right) P  \|_1 \leq C  \frac{1 + |\lambda|^2}{|\mathrm{Im}(\lambda^2)|^2}.
$$
\end{proposition}
\begin{proof}
We have that $P \tilde T_\lambda P = P T_{0,\lambda} P$
 and thus
 $$
  P(T_\lambda - T_{0,\lambda})P = P (T_\lambda - \tilde T_\lambda) P + P T_{0,\lambda} P.
 $$
 The bound now follows immediately from Lemma \ref{lemma32} and Lemma \ref{tracecornaught}.
 \end{proof}
 
\begin{proposition} \label{tracecor}
 For any even function $f \in \mathcal{S}(\R)$ we have that
 $$ 
  P\left(  Q \, f(\Delta_{\rel}^{1/2})  - Q\, f(\Delta_0^{1/2}) \right) P
 $$ is trace-class and the mapping
 $
  f \mapsto \Tr \left( P \left(Q \, f(\Delta_{\rel}^{1/2})   -  Q \, f(\Delta_0^{1/2}) \right) P \right)
 $
 is a tempered distribution.
\end{proposition}
\begin{proof}
Define $g \in \mathcal{S}(\R)$ by  $g (\lambda) = (1 + \lambda^2)^k f(\lambda)$.
Let $\tilde g$ be an almost analytic extension of $g$ such that $\frac{\partial \tilde g}{\partial \overline z} = O(|\mathrm{Im}(z)|^m)$ for some fixed $m \geq 5$.
Such an almost analytic extension can always be constructed as
$$
 \tilde g(x + \rmi y) = \sum_{k=0}^m \frac{1}{k!} g^{(k)}(x) (\rmi y)^k \chi(y),
$$
where $\chi \in C^\infty_\comp(\R)$ is chosen such that it equals one near $0$.
By the Helffer-Sj\"ostrand formula we have
 $$
  f(\Delta_{\rel}^{1/2}) Q = \frac{2}{\pi} \int\limits_{\mathrm{Im}(z)>0}z \frac{\partial \tilde g}{\partial \overline z} T_z \der m(z),
 $$
 and the analogous formula holds for $f(\Delta_0^{1/2})Q $. Here $\der m$ denotes the Lebesgue measure on $\C$. Hence,
 \begin{gather*}
   P\left(Q\, f(\Delta_{\rel}^{1/2})  - Q\, f(\Delta_0^{1/2})  \right) P \\= \frac{2}{\pi} \int\limits_{\mathrm{Im}(z)>0}\frac{\partial \tilde g}{\partial \overline z} P \left((\Delta_{\rel} +1)^{-k} Q\,  (\Delta_{\rel} - z^2)^{-1}  -  (\Delta_{0} +1)^{-k}  Q \, (\Delta_{0} - z^2)^{-1}  \right) P z \der m(z),
 \end{gather*}
 which implies the statement as the trace norm is finite and can be estimated as
 $$
  \Tr \left( | P Q\,  f(\Delta_{\rel}^{1/2})  -  Q\,  f(\Delta_0^{1/2}) P | \right)  \leq C \frac{2}{\pi} \int\limits_{\mathrm{Im}(z)>0} | \frac{\partial \tilde g}{\partial \overline z}|  \frac{1+| z |^2}{|\mathrm{Im}(z^2)|^2} \der m(z).
 $$
 \end{proof}
The same proof applied to Lemma \ref{tracecornaught} also gives
\begin{proposition} \label{tracenew}
 For any even function $f \in \mathcal{S}(\R)$ we have that
 $$ 
  (1-P) \left(  Q \, f(\Delta^{1/2}_{\rel}) \right)
 $$ is trace-class and the mapping
 $
  f \mapsto \Tr \left( (1-P) \left(Q \, f(\Delta^{1/2}_{\rel}) \right) \right)
 $
 is a tempered distribution.
\end{proposition} 
 
 We will now conclude with the proof of Theorem \ref{main1}, which will be a modification of the proof for $p$-forms in \cite{OS}, Theorem 6.1,  by carefully inserting suitable projection operators in the right places without destroying crucial positivity properties. The approach relies essentially on the Maass-Selberg technique, which requires a subtle analysis near $\lambda=0$. We provide here more details than in 
 \cite{OS} as we believe this result cannot be obtained from the abstract Birman-Krein formula.
\begin{proof}[Proof of Theorem \ref{main1}] It remains to show the trace-formula.
Since both sides are distributions in $\mathcal{D}'(\R)$ it suffices to check the trace formula for a dense class of functions. 
We will thus assume here that $f$ is real analytic in some neighborhood of zero, depending on $f$. We fix $0 \leq p \leq d$.
By Prop. \ref{kernelconv}, the operators $f(\Delta_{\rel}^{1/2})$ and $f(\Delta^{1/2}_0)$ have smooth integral kernels $k(x,y)$ and $k_0(x,y)$ respectively. We denote the integral kernels of $Q f(\Delta_{\rel}^{1/2})$ and $Q f(\Delta_0^{1/2})$
by $q(x,y)$ and $q_0(x,y)$.
Convergence in \ref{kernelconv} is in the space $C^\infty(M^\circ \times M^\circ; \Lambda^p T^*M \boxtimes (\Lambda^p T^*M)^*)$. In case $Q=\der \delta$ the kernel $q$ is the integral kernel of the map 
$\der f(\Delta_{\rel}^{1/2}) \delta : C^\infty_0(M^\circ,\Lambda^p T^*M) \to C^\infty(M^\circ,\Lambda^p T^*M)$, as $\overline \der$ and $\overline \der^*$ commute with $f(\Delta_{\rel}^{\frac{1}{2}})$. Similarly, if $Q=\delta \der$ the kernel $q$ is the integral kernel of $\delta f(\Delta^{1/2}) \der$ restricted to the space of $p$-forms.
By Prop. \ref{difftheo} and Prop. \ref{kernelconv} we therefore obtain
\begin{gather}
 q(x,y) =  \frac{1}{2\pi}  \sum\limits_{\nu}  \int_{0}^\infty \lambda^2 f(\lambda) E_{\lambda}(a_Q \Phi_\nu)(x) \otimes E_{\lambda}(a_Q  \Phi_\nu)^*(y)\,\der \lambda,
\end{gather}
where again the sum converges in $C^\infty(M^\circ \times M^\circ; \Lambda^p T^*M \boxtimes (\Lambda^p T^*M)^*)$. Here $a_Q$ is the operator of exterior multiplication by $\der r$ if $Q= \der \delta$, and it is the operator of interior multiplication by $\der r$ if $Q=\delta \der$. In particular, $P_Q=a_Q^* a_Q$ is the orthogonal projection onto $L^2_t(\sphere,\Lambda^\bullet \C^d)$ if $Q=\der \delta $ and it is the orthogonal projection onto $L^2_n(\sphere,\Lambda^\bullet \C^d)$ if $Q=\delta \der $. 

We define the family $(q_\nu)_\nu$ of smooth kernels $q_\nu \in C^\infty(M^\circ \times M^\circ; \Lambda^p T^*M \boxtimes (\Lambda^p T^*M)^*)$ by
\begin{gather*}
 q_\nu(x,y) =  \frac{1}{2\pi}   \int_{0}^\infty \lambda^2 f(\lambda) E_{\lambda}(a_Q \Phi_\nu)(x) \otimes E_{\lambda}(a_Q \Phi_\nu)^*(y)\,\der\lambda. 
\end{gather*}
In the same way we construct $q_{0,\nu} \in C^\infty(\R^d \times \R^d; \Lambda^p T^*\R^d \boxtimes (\Lambda^p T^* \R^d)^*)$ for $Q f(\Delta^{1/2}_0)$. The operator
  $(1-P) Q f(\Delta_{\rel}^{1/2}) (1-P)$ is trace-class by Proposition \ref{tracenew} and has smooth kernel on $\Omega_R$. 
  Since the trace of a trace-class operator can be computed as a limit of traces on any increasing sequence of subspaces 
  with union $L^2(\Omega_R)$ the trace equals  
  \begin{gather*}
    \Tr \left( (1-P) Q f(\Delta_{\rel}^{1/2}) (1-P) \right) = \lim_{n \to \infty} \int_{\Omega_{R,n}} \tr\, q(x,x) \der x,
  \end{gather*}
where $\Omega_{R,n}$ is an increasing compact exhaustion of $\Omega_{R}$ and $\tr$ denotes the pointwise trace on the fibre $\mathrm{End}(\Lambda^p T^*_x M)$ of $\Lambda^p T^*M \boxtimes (\Lambda^p T^*M)^*$ at the point $(x,x)$.
Note that for positive $f$ the kernels $\tr\; q(x,x)$ and $\tr\; q_\nu(x,x)$ are positive. Since $(1-P) Q f(\Delta_0^{1/2}) (1-P)$ is trace-class it follows that $\tr\; q(x,x)$ and $\tr\; q_\nu(x,x)$ are integrable on $M_{\rho}$. This also implies that
  $\tr\; q(x,x)$ and $\tr\; q_\nu(x,x)$ are integrable on $M_{\rho}$ for general $f$, since any Schwartz function can be dominated by a positive Schwartz function.
  We can therefore write
   \begin{gather*}
    \Tr \left( (1-P) Q f(\Delta_{\rel}^{1/2}) (1-P) \right) = \int_{\Omega_{R}} \tr\, q(x,x) \der x=\int_{\Omega_{R}}\sum_\nu  \tr\, q_\nu(x,x) \der x,
  \end{gather*}

  Similarly, the operator $(1-P_0) Q f(\Delta_0^{1/2}) (1-P_0)$ has smooth kernel on $\overline{B_R(0)}$ and is therefore trace-class with
 \begin{gather*}
  \Tr \left( (1-P_0) Q f(\Delta_0^{1/2}) (1-P_0) \right) = \int_{B_R}  \tr \, q_{0}(x,x) \der x= \int_{B_R} \sum_\nu   \tr \, q_{0,\nu}(x,x) \der x.
 \end{gather*}
  Now let $\chi_{\rho}$ be the indicator function of a large ball $B_\rho$ such that $\rho> R$. Then,
 $$
  \Tr \left( \chi_{\rho} \left (  P Q f(\Delta^{1/2}_{\rel}) P - P_0 Q f(\Delta_0^{1/2}) P_0  \right)\chi_{\rho} \right) = \int_{B_\rho \setminus B_R}  \tr\, (q(x,x)  - q_0(x,x)) \der x.
 $$ 
 By Proposition \ref{tracecor} the operator  $PQf(\Delta^{1/2})P - P_0Qf(\Delta_0^{1/2}) P_0$ is trace-class 
 and we can again compute its trace on an increasing sequence of subspaces exhausting the Hilbert space. This gives
  $$
  \Tr \left( P Q f(\Delta^{1/2}_{\rel})P - P_0 Q f(\Delta_0^{1/2}) P_0 \right) = \lim_{\rho \to \infty}  \int_{B_\rho \setminus B_R}  \tr\, (q(x,x)  - q_0(x,x)) \der x.
 $$
   
Collecting everything we have
 \begin{gather} \label{coolequ}
   \Tr \left( (1-P) Q f(\Delta^{1/2}_{\rel}) (1-P) \right) - \Tr \left( (1-P_0) Q f(\Delta_0^{1/2}) (1-P_0) \right) \\+ \Tr \left( P Q f(\Delta^{1/2}_{\rel})P - P_0 Q f(\Delta_0^{1/2}) P_0 \right)=  \nonumber
     \lim_{\rho \to \infty}  \sum_{\nu}\left(  \int_{M_{\rho}}  \tr\; q_\nu(x,x)  \der x -  \int_{B_\rho}  \tr\;  q_{0,\nu}(x,x)  \der x \right),
 \end{gather}
 where $M_\rho$ is obtained from $M$ by removing the subset identified with $\R^d \setminus B_\rho$. 
  It is common to use the following (Mass-Selberg-) trick to compute these integrals. Since $(\Delta_\rel - \lambda^2) E_\lambda(\Phi) =0$, differentiation in $\lambda$ yields $(\Delta_\rel - \lambda^2) E^{\;\prime}_\lambda(\Phi)  = 2 \lambda E_\lambda(\Phi)$, where $E^{\;\prime}_\lambda(\Phi) = \frac{d}{d\lambda} E_\lambda(\Phi)$. Note that $E_\lambda$ is differentiable for $\lambda>0$
  since it is analytic in $\lambda$ on $(0,\infty)$.
  Hence, integration by parts, Prop. \ref{ibpprop}, gives for a positive Schwartz function $f$ the equality
  \begin{gather*}
  \int_{M_\rho} \tr \;q_\nu(x,x) \der x= \lim_{\epsilon \to 0_+}\frac{1}{2 \pi} \int_{M_\rho}  \int_{\epsilon}^\infty \lambda^2 f(\lambda)  \langle E_\lambda(a_Q \Phi_\nu) ,E_\lambda(a_Q \Phi_\nu) \rangle \der \lambda \der x \\ =  \lim_{\epsilon \to 0_+}\frac{1}{4 \pi }  \int_{M_\rho} \int_{\epsilon}^\infty \lambda f(\lambda)  \langle  (\Delta_\rel - \lambda^2)  E^{\;\prime}_\lambda(a_Q \Phi_\nu),E_\lambda(a_Q \Phi_\nu) \rangle  \der \lambda \der x \\= \lim_{\epsilon \to 0_+}
   \frac{1}{4 \pi}  \int_{\epsilon}^\infty  \lambda f(\lambda)  b_\rho(E^{\;\prime}_\lambda(a_Q \Phi_\nu),  E_\lambda(a_Q \Phi_\nu) ) \der \lambda.
 \end{gather*}
 Here interchanging the order of integration is justified by Fubini's theorem. Similarly, interchanging limit and integration
 commute by Fatou's lemma.  
 Here $b_\rho(F,G)$ is the boundary pairing of forms $F$ and $G$ and defined by
 $$
  b_\rho(F,G) = \int_{\partial M_\rho} \langle F(x) ,\nabla_n G(x) \rangle  - \langle \nabla_n F(x),G(x) \rangle \der \sigma(x),
 $$
 where $\der \sigma$ is the surface measure of $\partial M_\rho$. Integration by parts is justified by Proposition \ref{ibpprop}. Indeed, differentiating Equation \eqref{Econstr} at $\lambda>0$ implies that for any smooth compactly supported cutoff function $\chi$ that equals one near $\partial \Omega$ the form $\chi E^{\;\prime}_\lambda(\Phi)$ is in the domain of $\Delta_{\rel}$.
 As before, by linearity, this implies that the equality holds without the assumption of positivity.

 We conclude that
 \begin{gather*}
 \int_{M_\rho} 
  \tr \;q_\nu(x,x)  \der x -\int_{B_\rho} 
  \tr \;q_{0,\nu}(x,x)\der x =  \lim_{\epsilon \to 0_+} \frac{1}{4 \pi} \int_{\epsilon}^\infty \lambda f(\lambda) \eta_{\nu,\rho}(\lambda) \der \lambda, \\
  \eta_{\nu,\rho}(\lambda) = b_\rho\left(\frac{d}{d \lambda}\left( \tilde j_\lambda(a_Q \Phi_\nu) + \tilde h^{(1)}_\lambda(A_\lambda a_Q \Phi_\nu) \right), \tilde j_{\overline{\lambda}}(a_Q \Phi_\nu) + \tilde h^{(1)}_{\overline{\lambda}}(A_{\overline{\lambda}} a_Q \Phi_\nu) \right) \\-  b_\rho\left( \frac{d}{d \lambda}\left( \tilde j_\lambda( a_Q\Phi_\nu) \right) , \tilde j_{\overline{\lambda}}(a_Q \Phi_\nu)\right)=
   b_\rho\left(\frac{d}{d \lambda}\left( \tilde j_\lambda(a_Q \Phi_\nu) + \tilde h^{(1)}_\lambda(A_{\lambda} a_Q \Phi_\nu) \right), \tilde h^{(1)}_{\overline{\lambda}}(A_{\overline{\lambda}} a_Q \Phi_\nu) \right)\\
   +b_\rho\left(\frac{d}{d \lambda}\left(\tilde h^{(1)}_\lambda(A_\lambda a_Q \Phi_\nu) \right), \tilde j_{\overline{\lambda}}(a_Q \Phi_\nu)  \right).
 \end{gather*}
 We have $\frac{d}{d\lambda} \left(\tilde h^{(1)}_\lambda(A_\lambda  \Phi) \right) =  \tilde h^{(1)}_\lambda(A^{\prime}_\lambda \Phi) +\tilde h^{(1)\prime}_\lambda(A_\lambda \Phi)$.
 Unitarity of $S(\lambda)$ implies the identity $A(\lambda) + A^*(\overline{\lambda}) + A^*(\overline{\lambda}) A(\lambda) =0$, and therefore
 \begin{gather*}
   b_\rho\left(\tilde h^{(1)\prime}_\lambda(A_\lambda a_Q \Phi_\nu),  \tilde h^{(1)}_{\overline\lambda}(A_{\overline{\lambda}} a_Q \Phi_\nu)\right) + b_\rho\left(\tilde h^{(1)\prime}_\lambda( q_Q \Phi_\nu),  \tilde h^{(1)}_{\overline{\lambda}}(A_{\overline{\lambda}} a_Q  \Phi_\nu)\right)  \\+b_\rho\left(\tilde h^{(1)\prime}_\lambda(A_\lambda a_Q \Phi_\nu),  \tilde h^{(1)}_{\overline{\lambda}}(a_Q \Phi_\nu)\right) =0.
 \end{gather*}
 By Prop. \ref{difftheo} the operators $A_\lambda$ and $A_\lambda'$ commute with $a_Q$. Note however that in
 $A_\lambda a_Q \Phi = a_Q A_\lambda \Phi$ the scattering matrices on the left and right hand side act on forms of different degree.
 Using $b_{\rho}( \tilde h^{(1)}_\lambda(\Phi_\nu), \tilde h^{(2)}_{\overline{\lambda}}(\Phi_\nu))=0$ and $\tilde j_\lambda(\Phi_\nu)=  \tilde h^{(1)}_\lambda(\Phi_\nu)+ \tilde h^{(2)}_\lambda(\Phi_\nu)$ and the fact that integration over the sphere results in only diagonal terms with respect to the basis $(\Phi_\nu)$
  one obtains
  \begin{gather} \label{lineone}
  \eta_{\nu,\rho}(\lambda)=\left( \langle A_\lambda^{\prime} P_Q \Phi_\nu, A_{\overline{\lambda}} \Phi_\nu \rangle + \langle A_\lambda^{\prime} P_Q \Phi_\nu,\Phi_\nu\rangle \right) b_{\rho}( \tilde h^{(1)}_\lambda(\Phi_\nu), \tilde h^{(1)}_{\overline{\lambda}}(\Phi_\nu))\\+
  \langle  P_Q \Phi_\nu, A_{\overline{\lambda}} \Phi_\nu\rangle b_{\rho}( \tilde h^{(2) \prime}_\lambda(\Phi_\nu), \tilde h^{(1)}_{\overline{\lambda}}(\Phi_\nu))+
  \langle  A_\lambda P_Q \Phi_\nu, \Phi_\nu\rangle b_{\rho}(\tilde h^{(1) \prime}_\lambda(\Phi_\nu), \tilde h^{(2)}_{\overline{\lambda}}(\Phi_\nu)). \nonumber
   \end{gather}
 The term $b_{\rho}( \tilde h^{(1)}_\lambda(\Phi_\nu), \tilde h^{(1)}_{\overline{\lambda}}(\Phi_\nu))$ is independent of $\rho$ and is actually given in terms of a Wronskian between Hankel functions. One obtains
 $$
 b_{\rho}( \tilde h^{(1)}_\lambda(\Phi_\nu), \tilde h^{(1)}_{\overline{\lambda}}(\Phi_\nu)) = -2 \rmi \lambda.
 $$
 This first summand in \eqref{lineone} therefore equals
 $-2 \rmi \lambda \langle S_\lambda^* S_\lambda'  P_Q \Phi_\nu, \Phi_\nu \rangle$. In case $Q=\delta \der$ we need to bear in mind that
 $\Phi_\nu$ has form degree $p+1$, whereas in case $Q=\der \delta$ it has form degree $p-1$.
We define $S_Q(\lambda):= S_{t,\lambda}$ if $Q=\delta \der$, and $S_Q(\lambda):= S_{n,\lambda}$ in case $Q=\der \delta$. Then
$$
  \sum_\nu \langle S_\lambda^* S_\lambda'  P_Q \Phi_\nu, \Phi_\nu \rangle = \Tr_{L^2(\sphere,\Lambda^{p \pm 1} \C^d)}(S_\lambda^* S_\lambda'  P_Q ) = \Tr_{L^2(\sphere,\Lambda^p \C^d)} \left ( S_Q^*(\lambda) S_Q^{\;\prime}(\lambda) \right).
$$
\begin{lemma}\label{odd}
The terms 
 $$g(\lambda) := \langle  A_\lambda P_{Q} \Phi_\nu, \Phi_\nu\rangle b_{\rho}(\tilde h^{(1) \prime}_\lambda(\Phi_\nu), \tilde h^{(2)}_{\overline \lambda}(\Phi_\nu))$$ and $$ \langle P_Q\Phi_\nu, A_{\overline \lambda} \Phi_\nu\rangle b_{\rho}(\tilde h^{(2)}_\lambda(\Phi_\nu), \tilde h^{(1) \prime}_{\overline \lambda}(\Phi_\nu))$$ are complex conjugates of each other for positive $\lambda$. As long as $\lambda>0$ their sum is 
 $$
 2 \Re\left(\langle  A_\lambda P_Q \Phi_\nu, \Phi_\nu\rangle b_{\rho}(\tilde h^{(1)}_\lambda(\Phi_\nu), \tilde h^{(2) \prime}_{\overline \lambda}(\Phi_\nu))\right) = 2 \Re(g(\lambda)).$$
 Moreover $\Re g(\lambda)$ is odd in the sense that $\Re g(-\lambda) = -\Re g(\lambda)$ for $\lambda>0$.
\end{lemma}
\begin{proof} 
First note that  $\tau\Phi_{\nu}=(-1)^{\ell_{\nu}}\Phi_{\nu}$. Then and Definition \eqref{hankelsums}, along with the equations \eqref{rotationhankel1} and \eqref{rotationhankel2}  also imply that
\begin{align*}
&\overline{\tilde{h}_{\lambda}^{(1)}(\Phi_{\nu})}=(-1)^{\ell_{\nu}}\tilde{h}_{\overline{\lambda}}^{(2)}(\Phi_{\nu}) \quad \overline{\tilde{h}_{\lambda}^{(2)}(\Phi_{\nu})}=(-1)^{\ell_{\nu}}\tilde{h}_{\overline{\lambda}}^{(1)}(\Phi_{\nu})
\\& \tilde h^{(1)}_{-\lambda}(\Phi_{\nu})=\rmi^{d-1} (-1)^{\ell_\nu+d+1} \tilde h^{(2)}_\lambda(\Phi_{\nu}).
\end{align*}
Using the functional equation for $A_\lambda$ given by Proposition \ref{Aone} (vi), and combining the equations mentioned above, we can show that the function
 $g$
 satisfies $\overline{g(-\lambda)} = - g(\lambda)$. Thus, $\Re(g(\lambda))$ is odd. 
 \end{proof} 
Using the Lemma \ref{odd} one can then change the domain of integration 
 \begin{gather*}
    \lim_{\epsilon \to 0_+} \frac{1}{4 \pi} \int_{\epsilon}^\infty \lambda f(\lambda) 2 \Re\left( \langle P_QA_\lambda  \Phi_\nu, \Phi_\nu\rangle b_{\rho}(\tilde h^{(1) \prime}_\lambda(\Phi_\nu), \tilde h^{(2)}_{\overline \lambda}(\Phi_\nu))\right)  \der \lambda \\
   =\Re \lim_{\epsilon \to 0_+} \frac{1}{4 \pi} \int_{\R_\epsilon}\lambda f(\lambda) \langle P_QA_\lambda  \Phi_\nu, \Phi_\nu\rangle b_{\rho}(\tilde h^{(1) \prime}_\lambda(\Phi_\nu), \tilde h^{(2)}_{\overline \lambda}(\Phi_\nu))  \der \lambda.
   \end{gather*} 
   
 Summarising, we have
 \begin{align}\label{tracequail} 
 & \Tr \left( (1-P) Q f(\Delta^{1/2}_{\rel}) (1-P) \right) - \Tr \left( (1-P_0) Q f(\Delta_0^{1/2}) (1-P_0) \right)\nonumber \\&+ \Tr \left( P Q f(\Delta^{1/2}_{\rel})P - P_0 Q f(\Delta_0^{1/2}) P_0 \right)  \\  &= \frac{1}{2 \pi \rmi }\int_{0}^\infty\lambda^2  f(\lambda) \Tr_{L^2(\sphere,\Lambda^p \C^d)} \left ( S_Q^*(\lambda) S_Q^{\;\prime}(\lambda) \right) \der \lambda \nonumber \\&+  \lim_{\rho \to \infty} \sum_\nu \Re \lim_{\epsilon \to 0_+}\frac{1}{4 \pi} \int_{\R_\epsilon}\lambda f(\lambda) \langle P_QA_\lambda  \Phi_\nu, \Phi_\nu\rangle b_{\rho}(\tilde h^{(1) \prime}_\lambda(\Phi_\nu), \tilde h^{(2)}_{\overline \lambda}(\Phi_\nu))  \der \lambda. \nonumber
 \end{align}
It now remains to only show that the second term of the right hand side of the equation vanishes. 
Note that the function $b_{\rho}(\tilde h^{(1) \prime}_\lambda(\Phi_\nu), \tilde h^{(2)}_\lambda(\Phi_\nu))$ depends only on $\ell_\nu$ and $\lambda \rho$.
We can therefore define $H_\ell$ by  $H_\ell(\lambda \rho)= b_{\rho}(\tilde h^{(1) \prime}_\lambda(\Phi_\nu), \tilde h^{(2)}_\lambda(\Phi_\nu))$. 
For the sake of completeness, we recall Lemma 6.2 in \cite{OS}: 
\begin{lemma}[Lemma 6.2 in \cite{OS}] \label{superhankel}
 Let as before $H_\ell(\lambda \rho):= b_{\rho}(\tilde h^{(1) \prime}_\lambda(\Phi_\nu), \tilde h^{(2)}_\lambda(\Phi_\nu))$. 
 Suppose that $\tilde{f} \in C^\infty_0(\R)$ is supported in $(-T,T)$ and extends holomorphically near zero to a function analytic in a neighborhood of the closed ball $\overline{B_\delta(0)}$.
 Let $\mathrm{Rec}:=[-T,T] \times [0, \delta_1] \subset \C$ be any rectangle with $\delta_1>0$.
 Then for every $k \in \mathbb{N}$ there exists a constant $C_k>0$, independent of $\nu$ such that for any $\rho>\delta^{-1}$ and any $g$ that is holomorphic in the interior of $\mathrm{Rec}$ and continuous on $\mathrm{Rec}$ we have the following estimates for $\rho>1$;
   \begin{itemize}
  \item if $d=2$ and $\ell_\nu=0$  then 
  $$
 | \lim_{\epsilon \to 0_+} \int_{\R_\epsilon} \frac{1}{\lambda} \tilde{f}(\lambda) g(\lambda) H_\ell(\lambda \rho) \der \lambda - (-2 \, \rmi \,g(0))  \tilde{f}(0) | \leq  \frac{C_k}{\rho^k} \sup\limits_{x \in \mathrm{Rec}} |g(x)| $$
  \item if $d=2$ and $\ell_\nu=1$ and $g(\lambda) = \frac{a}{-\log \lambda}+ o(\frac{1}{-\log \lambda})$ for $|\lambda|<1/2$ then 
  $$
  | \lim_{\epsilon \to 0_+} \int_{\R_\epsilon} \frac{1}{\lambda} \tilde{f}(\lambda) g(\lambda) H_\ell(\lambda \rho) \der \lambda -(4 \, \rmi \, a) \tilde{f}(0) | \leq \frac{C_k}{\rho^k} \sup\limits_{x \in \mathrm{Rec}} |g(x)|  $$
  \item  if $d=3$ and $\ell_\nu=0$ then 
   $$
  | \lim_{\epsilon \to 0_+} \int_{\R_\epsilon}\frac{1}{\lambda} \tilde{f}(\lambda) g(\lambda) H_\ell(\lambda \rho) \der \lambda - (-\pi g(0)) \tilde{f}(0)| \leq\frac{C_k}{\rho^k} \sup\limits_{x \in \mathrm{Rec}} |g(x)|  .
 $$
 \item if $2\ell +(d-4)>0$ and $g(\lambda) = a \lambda^{2 \ell +  d - 4}  +o(\lambda^{2 \ell +  d - 4}) $  for $|\lambda|<1$ then
  \begin{gather*}
|\lim_{\epsilon \to 0_+} \int_{\R_\epsilon} \frac{1}{\lambda} \tilde{f}(\lambda) g(\lambda) H_\ell(\lambda \rho) \der \lambda - a\tilde{f}(0) \gamma_{d,\ell} \rho^{-2\ell- d+4}| \\ \leq \frac{C_k (1+\ell)^2}{\rho^k} \sup\limits_{x \in \mathrm{Rec}} |g(x)|  e^{2(1+\frac{d}{2})^2 \rho^{-1} \delta^{-1}  (1+\ell)^2},
 \end{gather*}
 where $\gamma_{d,\ell} =  \rmi\; 2^{2\ell+d-3} \Gamma(\ell +\frac{d-2}{2}) \Gamma(\ell +\frac{d-5}{2})$.  
 \end{itemize}
\end{lemma}
We apply this Lemma with $\tilde{f}(\lambda)=\lambda^2 f(\lambda)$ and  $\langle  A_\lambda P_Q \Phi_\nu, \Phi_\nu\rangle=g_\nu(\lambda)$. By Lemma \ref{distrellemmar} the form $P_Q \Phi_\nu$ decomposes into a linear combination
of spherical harmonics of degrees $\ell_\mu -2, \ell_\mu$ and $\ell_{\mu}+2$. By \ref{Aone} (iv) these functions satisfy the required bounds, and we obtain:
 $$
  |\lim_{\epsilon \to 0_+} \frac{1}{4 \pi} \int_{\epsilon}^\infty \lambda f(\lambda) \langle A_\lambda P_Q \Phi_\nu, \Phi_\nu\rangle b_{\rho}(\tilde h^{(1) \prime}_\lambda(\Phi_\nu), \tilde h^{(2)}_{\overline{\lambda}}(\Phi_\nu))  \der \lambda | \leq C_1 \rho^{-1} \sup\limits_{x \in \mathrm{Rec}} |g_\nu (x)|.
 $$
 Here we use that $\tilde f(0)=0$.  By Proposition \ref{Aone} (v) 
 we have the bound
 $$
  \sup\limits_{x \in \mathrm{Rec}} |g_\nu(x)|  \leq C_2 \left(R_1^{\ell_\mu+\ell_\nu+2} \frac{\lambda^{\ell_\nu + \ell_\mu + d-4}}{\Gamma(\ell_\nu +\frac{d}{2}) \Gamma(\ell_\mu +\frac{d-2}{2})}\right).
 $$
 The decay of the right hand side then gives absolute convergence of the sum and the bound
 $$
   \sum_\nu \lim_{\epsilon \to 0_+} \frac{1}{4 \pi} \int_{\epsilon}^\infty \lambda f(\lambda) \langle A_\lambda P_Q \Phi_\nu, \Phi_\nu\rangle b_{\rho}(\tilde h^{(1) \prime}_\lambda(\Phi_\nu), \tilde h^{(2)}_{\overline{\lambda}}(\Phi_\nu))  \der \lambda \leq \frac{C_3}{\rho}.
 $$
 Since this converges to zero as $\rho \to \infty$ the second term on the right hand side of Equ. \eqref{tracequail} vanishes as claimed. The proof is complete.
\end{proof}

\begin{proof}[Proof of Theorem \ref{B2}]

We have already established in Proposition \ref{tracecorone} that the operator
$$
  T_f = \curl\,\curl f(\Delta_{\rel,\R^3}^{1/2}) -  \curl\,\curl f(\Delta_0^{1/2})
$$
is trace-class. Let $p_R$ be the projection onto $L^2(B_R \setminus \overline{\Omega})$. 
The operator 
$$
 (1-p_R) \curl\,\curl f(\Delta_{\rel,\R^3}^{1/2}) (1-p_R)
$$
is trace-class and its trace equals
$$
 \sum_{j=1}^\infty \mu_j^2 f(\mu_j).
$$
Then, by Theorem \ref{main1}
we have
$$
 p_R \curl\,\curl f(\Delta_{\rel,\R^3}^{1/2}) p_R - \curl\,\curl f(\Delta_0^{1/2})
$$
is trace-class and its trace equals 
$$
\frac{1}{2 \pi \rmi }\int_{0}^\infty\lambda^2  f(\lambda) \Tr_{L^2(\mathbb{S}^2,\C^3)} \left ( S_t^*(\lambda) S_t^{\;\prime}(\lambda) \right) \der \lambda.
$$
Simply adding these two terms gives the theorem.
\end{proof}

%

\appendix

\section{The relative Laplace operator} \label{relap}

As explained in the introduction the operator $\Delta_{\rel}$ is defined on a general manifold $Z$ with Hermitian inner product on $\Lambda^\bullet T^*Z$ as 
$$
 \Delta_{\rel} = \overline{\der}\,\overline{\der}^* + \overline{\der}^*\overline{\der},
$$
with its natural domain 
$$
 \mathrm{dom}(D^2)=\{ f \in \mathrm{dom}(\overline{\der}) \cap \mathrm{dom}(\overline{\der}^*) \,\mid \,  \overline{\der} f \in  \mathrm{dom}(\overline{\der}^*),  \overline{\der}^* f \in  \mathrm{dom}(\overline{\der}) \}.
$$
The fact that this operator is self-adjoint is a consequence of a more general abstract statement about closed operators. 
Self-adjointness of Laplace-operators defined in a similar way was shown by Gaffney (\cite{MR68888}). As pointed out by Kohn (\cite{MR153030}) his method also applies to the relative Laplacian and has been used in the literature in various forms to prove self-adjointness (for example \cite{MR0461588,MR929141}).
An abstract statement, using Gaffney's proof, can be found in \cite{MR2463962}.  Similarly, such constructions also appear in \cite{lesch} in the context of Hilbert complexes.

For the convenience of the reader we give here a short and direct proof of self-adjointness of $\Delta_{\rel}$ with a slightly more refined conclusion. The Lemma below is formulated for generic operators $T$. 
\begin{lemma}\label{Tstar}
 Suppose that $T_0$ is a densely defined operator in a Hilbert space $H$ with densely defined adjoint $T_0^*$. Let
 $T = \overline{T} = T_0^{**}$ be its closure and $T^* = T_0^*$ its adjoint. Assume that $\mathrm{rg}(T_0) \subset \ker(T_0)$
 and that $\dom(T) \cap \dom(T^*)$ is dense.
 Then the following statements hold
 \begin{enumerate}
 \item the operator $T T^* + T^* T$ is densely defined and self-adjoint.
 \item The closure of the quadratic form 
 $q(f,f) = \langle T f, Tf \rangle + \langle T^* f, T^* f \rangle$ 
 defined on $\dom(T_0) \cap \dom(T^*)$ has associated self-adjoint operator $T^* T + T T^*$. In other words 
 $T^* T + T T^*$ is the Friedrichs extension of the symmetric operator $T_0 T^* + T^* T_0$.
  \item 
 The operators $T^* T$ and $T^* T$ form a commuting pair of self-adjoint operators in the sense that there exists a joint spectral resolution.
 \item the operator $T+ T^*$ is self-adjoint, its square equals $T^* T + T T^*$ and 
 $$ \| T T^* (T T^* + T^* T + 1)^{-1} \| \leq 1, \quad \| T^* T (T T^* + T^* T + 1)^{-1}\| \leq 1.$$
 \end{enumerate}
\end{lemma}
\begin{proof}
Since $T$ is closed and densely defined we have the orthogonal sum decomposition $H=H_1  \perp H_2$, where  $H_1=\mathrm{ker}(T)$ and  $H_2=\overline{\mathrm{rg}(T^*)}$. 
We have of course $H_1 \subset \dom(T)$ and $T |_{H_1}=0$. Moreover, by assumption, $\mathrm{rg}(T_0) \subset \ker(T_0) \subset  \ker(T)$, which implies $\overline{\mathrm{rg}(T)} = \overline{\mathrm{rg}(T_0)} \subset \ker(T)$.
Let $S$ be the restriction of $T$ to $\dom(T) \cap H_2$. The above means that $\mathrm{rg}(S) \subset H_1$.
Identifying $H$ with $H_1 \oplus H_2$ we have
$$
 T = \left( \begin{matrix} 0 & S \\ 0 & 0 \end{matrix} \right), \textrm{ and therefore}\quad T^* = \left(\begin{matrix} 0 & 0 \\ S^* & 0 \end{matrix} \right).
$$
It follows automatically that
$$
 T + T^* = \left( \begin{matrix} 0 & S \\ S^* & 0 \end{matrix} \right)
$$
is self-adjoint with domain $\dom(S^*) \oplus \dom(S) = \dom(T) \cap \dom(T^*)$. 
Its square
$$
 (T+ T^*)^2 = \left( \begin{matrix} S S^* & 0 \\ 0 & S^* S \end{matrix} \right) = T T^* + T^* T
$$
is therefore also self-adjoint. Since this represents  $T T^* + T^* T$ as a direct sum of the self-adjoint operators
$T^* T$ and $T T^*$ this shows that $T^* T$ and  $T T^*$  are both self-adjoint and commute with each other in the sense that their resolvents commute.\\
It is immediately clear that the space $\dom(T) \cap \dom(T^*)$ is complete with respect to the norm
$v \mapsto (\| T v \|^2 + \| T^* v \|^2 + \| v\|^2)^\frac{1}{2}$ therefore the quadratic form $q$ is closed.
We have 
$$
 q(v,w) = \langle T T^* + T^* T v, w \rangle = \langle (T + T^*)^2 v, w \rangle 
$$ 
for all $v \in \dom( T T^* + T^* T ) \subset \dom(T) \cap \dom(T^*)$. Uniqueness of the Friedrich's extension 
shows that it must be equal to  $T T^* + T^* T$.
\end{proof}

\section{Vector-valued spherical harmonics} \label{vecspher}

Let $H_\ell(\sphere,\Lambda^\bullet \C^d) = H_\ell(\sphere) \otimes \Lambda^\bullet \C^d$ be the space of $\Lambda^\bullet \C^d$-valued spherical harmonics
on $\sphere$. By restriction to the sphere this space can be identified with the space with the space $\mathcal{H}_\ell(\sphere,\Lambda^\bullet \C^d)$ of $\Lambda^\bullet \C^d$-valued harmonic homogeneous polynomials of degree $\ell$. 
The algebraic linear hull
$$
\mathcal{E} = \oplus_{\ell=0}^\infty H_\ell(\sphere,\Lambda^\bullet \C^d)
$$
is a dense subset of $C^\infty(\sphere,\Lambda^\bullet \C^d)$ and therefore also of $L^2(\sphere,\Lambda^\bullet \C^d)$. Recall that $L^2(\sphere,\Lambda^\bullet \C^d) = L^2_t(\sphere,\Lambda^\bullet \C^d) \oplus L^2_n(\sphere,\Lambda^\bullet \C^d)$, where $L^2_t(\sphere,\Lambda^\bullet \C^d)$ is the space of tangential differential forms and can be identified with the $L^2(\sphere, \Lambda^\bullet \sphere)$. Define $\theta =  \der r \wedge$ as the operator of exterior multiplication by $\der r$ and $\theta^* = \iota_{\der r}$ the operator of interior multiplication by $\der r $. Let $P_\theta= \theta^* \theta$, so that $1-P_\theta =  \theta \theta^*$. Then $P_\theta$ is the orthogonal projection onto $L^2_t(\sphere,\Lambda^\bullet \C^d)$ and $(1-P_\theta)$ is the orthogonal projection onto $L^2_n(\sphere,\Lambda^\bullet \C^d)$.
The following Lemma is certainly well known in dimension three and can in general be deduced from the representation theory of $\mathrm{SO}(d)$. We will provide a very short algebraic proof.
\begin{lemma} \label{distrellemmar}
 The operator $P_\theta$ leaves $\mathcal{E}$ invariant and maps $H_\ell(\sphere,\Lambda^\bullet \C^d)$ to 
 $$H_{\ell-2}(\sphere,\Lambda^\bullet \C^d) \oplus H_{\ell}(\sphere,\Lambda^\bullet \C^d) \oplus H_{\ell+2}(\sphere,\Lambda^\bullet \C^d),$$
where we define $H_\ell(\sphere,\Lambda^\bullet \C^d)=0$ if $\ell<0$.
\end{lemma}
\begin{proof}
 We will show that individually both $\theta$, and $\theta^*$ map $H_{\ell}(\sphere,\Lambda^\bullet \C^d)$ to $H_{\ell-1}(\sphere,\Lambda^\bullet \C^d) \oplus H_{\ell+1}(\sphere,\Lambda^\bullet \C^d)$, then the statement follows immediately.
 Let $\mathcal{A}$ be the $\mathbb{Z}$-graded algebra of polynomials  $\C[x_1,\ldots,x_d]$, graded by homogeneity, and denote that by $\mathcal{A}_\ell$ the degree $\ell$ subspace.
 The ring of polynomial functions on the sphere can be identified with the quotient $\mathcal{A}/\mathcal{I}$, where $\mathcal{I}$ is the ideal generated by 
 the polynomial $| x |^2-1 := x_1^2 + \ldots + x_d^2 - 1$. Then $H_\ell = \mathcal{H}_{\ell} / \mathcal{I}$.
 Recall from the theory of harmonic polynomials that
 $$
  \mathcal{A}_\ell = \bigoplus_{2k \leq \ell} | x |^{2 k} \mathcal{H}_{\ell-2k}.
 $$
 This shows that $\mathcal{E} = \mathcal{A}/\mathcal{I}$. 
 On the sphere $r=1$ and therefore $\der r = r \der r = \sum_{k=1}^dx^k \der x^k$, which makes sense on polynomial functions. It is therefore sufficient to show that multiplication of $x^j$ induces a map
 from $H_{\ell}(\sphere)$ to $H_{\ell-1}(\sphere) \oplus H_{\ell+1}(\sphere)$. This follows from
 $$
  x^j \mathcal{H}_\ell \subset \mathcal{H}_{\ell+1} + |x|^2 \mathcal{H}_{\ell-1}, 
 $$
 which we now show directly. Given $p \in \mathcal{H}_\ell$ an elementary computation shows that
 $$
  x^j p - \frac{1}{d + 2\ell - 2} |x|^2 \partial_j p
 $$
 is harmonic, if $d + 2\ell - 2>0$.  In case $d + 2\ell - 2=0$, the polynomial $x^j p$ is harmonic without subtraction.
 Since $\partial_j p$ is also harmonic, this completes the proof.
\end{proof}


\begin{bibdiv}
\begin{biblist}

\bib{MR928156}{article}{
   author={Birman, M. Sh.},
   author={Solomyak, M. Z.},
   title={Weyl asymptotics of the spectrum of the Maxwell operator for
   domains with a Lipschitz boundary},
   language={Russian, with English summary},
   journal={Vestnik Leningrad. Univ. Mat. Mekh. Astronom.},
   date={1987},
   number={vyp. 3},
   pages={23--28, 127},
}

\bib{MR929141}{article}{
   author={Borisov, N. V.},
   author={M\"{u}ller, W.},
   author={Schrader, R.},
   title={Relative index theorems and supersymmetric scattering theory},
   journal={Comm. Math. Phys.},
   volume={114},
   date={1988},
   number={3},
   pages={475--513},
}

\bib{lesch}{article}{
   author={Br\"uning, J.},
   author={Lesch, M.},
   title={Hilbert Complexes},
   journal={Journal of Functional Analysis},
   volume={108},
   date={1992},
   pages={88-132},
}


\bib{carron2003l2}{article}{
   author={Carron, G.},
   title={$L^2$-cohomology of manifolds with flat ends},
   language={English, with English and French summaries},
   journal={Geom. Funct. Anal.},
   volume={13},
   date={2003},
   number={2},
   pages={366--395},
}

\bib{DZ}{book}{
   author={Dyatlov, S.},
   author={Zworski, M.},
   title={Mathematical theory of scattering resonances},
   series={Graduate Studies in Mathematics},
   volume={200},
   publisher={American Mathematical Society, Providence, RI},
   date={2019},
   pages={xi+634},
}

\bib{MR3113431}{article}{
   author={Filonov, N.},
   title={Weyl asymptotics of the spectrum of the Maxwell operator in
   Lipschitz domains of arbitrary dimension},
   language={Russian},
   journal={Algebra i Analiz},
   volume={25},
   date={2013},
   number={1},
   pages={170--215},
   issn={0234-0852},
   translation={
      journal={St. Petersburg Math. J.},
      volume={25},
      date={2014},
      number={1},
      pages={117--149},
      issn={1061-0022},
   },
}

\bib{MR0461588}{book}{
   author={Folland, G.~B.},
   author={Kohn, J.~J.},
   title={The Neumann problem for the Cauchy-Riemann complex},
   note={Annals of Mathematics Studies, No. 75},
   publisher={Princeton University Press, Princeton, N.J.; University of
   Tokyo Press, Tokyo},
   date={1972},
   pages={viii+146},
}



\bib{MR68888}{article}{
   author={Gaffney, M.~P.},
   title={Hilbert space methods in the theory of harmonic integrals},
   journal={Trans. Amer. Math. Soc.},
   volume={78},
   date={1955},
   pages={426--444},
}

\bib{MR2839867}{article}{
   author={Gol'dshtein,~V.},
   author={Mitrea, I.},
   author={Mitrea, M.},
   title={Hodge decompositions with mixed boundary conditions and
   applications to partial differential equations on Lipschitz manifolds},
   note={Problems in mathematical analysis. No. 52},
   journal={J. Math. Sci. (N.Y.)},
   volume={172},
   date={2011},
   number={3},
   pages={347--400},
}


\bib{RT}{article}{
  author ={Hanisch, F.}
   author={Strohmaier, A.},
   author={Waters, A.},
   title={A relative trace formula for obstacle scattering},
   journal={https://arxiv.org/abs/2002.07291},
}

\bib{simon}{article}{
   author={Hempel, R.},
   author={Seco, L. A.},
   author={Simon, B.},
   title={The essential spectrum of Neumann Laplacians on some bounded
   singular domains},
   journal={J. Funct. Anal.},
   volume={102},
   date={1991},
   number={2},
   pages={448--483},
}

\bib{MR153030}{article}{
   author={Kohn, J. J.},
   title={Harmonic integrals on strongly pseudo-convex manifolds. I},
   journal={Ann. of Math. (2)},
   volume={78},
   date={1963},
   pages={112--148},
   issn={0003-486X},
}
		

\bib{MR3288313}{book}{
   author={Kirsch, A.},
   author={Hettlich, F.},
   title={The mathematical theory of time-harmonic Maxwell's equations},
   series={Applied Mathematical Sciences},
   volume={190},
   note={Expansion-, integral-, and variational methods},
   publisher={Springer, Cham},
   date={2015},
   pages={xiv+337},
}
		

\bib{MR2463962}{article}{
   author={Mitrea, D.},
   author={Mitrea, M.},
   author={Shaw, M.},
   title={Traces of differential forms on Lipschitz domains, the boundary de
   Rham complex, and Hodge decompositions},
   journal={Indiana Univ. Math. J.},
   volume={57},
   date={2008},
   number={5},
   pages={2061--2095},
}

\bib{MR2361423}{article}{
   author={Sanada, M.},
   title={Strong unique continuation property for some second order elliptic
   systems},
   journal={Proc. Japan Acad. Ser. A Math. Sci.},
   volume={83},
   date={2007},
   number={7},
   pages={119--122},
}



\bib{MR1852334}{book}{
   author={Shubin, M.~A.},
   title={Pseudodifferential operators and spectral theory},
   edition={2},
   note={Translated from the 1978 Russian original by Stig I. Andersson},
   publisher={Springer-Verlag, Berlin},
   date={2001},
   pages={xii+288},
}


\bib{OS}{article}{
   author={Strohmaier, A.},
   author={Waters, A.},
   title={Geometric and obstacle scattering at low energy},
   journal={Comm. Partial Differential Equations},
   volume={45},
   date={2020},
   number={11},
   pages={1451--1511},
}



\bib{MR1451399}{article}{
   author={Sj\"{o}strand, J.},
   title={A trace formula and review of some estimates for resonances},
   conference={
      title={Microlocal analysis and spectral theory},
      address={Lucca},
      date={1996},
   },
   book={
      series={NATO Adv. Sci. Inst. Ser. C Math. Phys. Sci.},
      volume={490},
      publisher={Kluwer Acad. Publ., Dordrecht},
   },
   date={1997},
   pages={377--437},
}

\bib{MR1180965}{book}{
   author={Yafaev, D. R.},
   title={Mathematical scattering theory},
   series={Translations of Mathematical Monographs},
   volume={105},
   note={General theory;
   Translated from the Russian by J. R. Schulenberger},
   publisher={American Mathematical Society, Providence, RI},
   date={1992},
   pages={x+341},
   isbn={0-8218-4558-6},
   review={\MR{1180965}},
   doi={10.1090/mmono/105},
}
		


\end{biblist}
\end{bibdiv}
\end{document}